\documentclass[12pt,twoside]{article}
\usepackage{amsfonts,amssymb,amsmath,url,bm}
\usepackage{a4wide}
\numberwithin{equation}{section}

\newcommand{\Supp}{\supp\hspace{1pt}\mu}
\newcommand{\rI}{{\rm I}}
\newcommand{\rH}{{\rm H}}
\renewcommand{\hom}{{\rm H}}
\newcommand{\inh}{{\rm I}}
\newcommand{\Cta}{{\cC}_{\ttt,\alpha}}

\newcommand\amr{$Y\in M_{m,n}$}
\newcommand\da{Diophantine approximation}
\newcommand\di{Diophantine}

\newcommand{\GL}{\operatorname{GL}}
\newcommand{\SL}{\operatorname{SL}}
\newcommand{\Span}{\operatorname{span}}

\newcommand{\fa}{{\mathcal A}}

\newcommand{\vf}{{\bf f}}
\newcommand{\vt}{{\bf t}}
\newcommand{\vp}{{\bf p}}
\newcommand{\vq}{{\bf q}}
\newcommand{\vw}{{\bf w}}
\newcommand{\ignore}[1]{}

\newcommand{\dd}{\mathbf{d}}
\newcommand{\supp}{{\operatorname{supp}}}

\newcommand{\bnz}{\smallsetminus\{\vv0\}}

\newcommand{\fw}{{\mathcal W}}
\newcommand{\df}{{\, \stackrel{\mathrm{def}}{=}\, }}
\newtheorem{theorem}{Theorem}[section]

\newtheorem{corollary}[theorem]{Corollary}
\newtheorem{lemma}[theorem]{Lemma}

\newtheorem{proposition}[theorem]{Proposition}

\def\R{\mathbb{R}}
\def\C{\mathbb{C}}

\def\Q{\mathbb{Q}}
\def\Z{\mathbb{Z}}
\def\N{\mathbb{N}}

\newcommand{\Rp}{\R^+}    

\def\cA{\mathcal{A}}

\def\cE{\mathcal{E}}
\def\cW{\mathcal{W}}
\def\cH{\mathcal{H}}

\def\cM{\mathcal{M}}

\def\cC{\mathcal{C}}

\newcommand{\eg}{{\it e.g.}}

\newcommand{\ve}{\varepsilon}

\newcommand{\vv}[1]{{\mathbf{#1}}}

\newcommand{\Id}{\operatorname{Id}}

\newcommand{\rank}{\operatorname{rank}}

\newcommand{\diag}{\operatorname{diag}}
\newenvironment{proof}{\noindent\textit{{Proof}}.}{\hspace*{\fill}\raisebox{-1ex}{$\boxtimes$}}

\newcommand{\proofend}{{\hspace*{\fill}\raisebox{-2ex}{$\boxtimes$}}}
\renewcommand{\tilde}{\widetilde}

\newcommand{\PI}{{\textstyle\prod}}

\newcounter{remark}
\newcounter{example}
\newcounter{question}
\newcounter{problem}
\newcounter{conjecture}
\newcommand{\TTT}{\mathbf{T}}
\newcommand{\ttt}{\mathbf{t}}
\newcommand{\AAA}{\mathcal{A}}

\newcommand{\La}{\Lambda}
\newcommand{\p}{\psi}

\newcommand{\qqand}{\qquad\text{and}\qquad}
\newcommand{\qand}{\quad\text{and}\quad}

\newcommand{\Mat}[2]{M_{#1,#2}}

\newcommand\eq[2]{\begin{equation}\label{eq:#1}{#2}\end{equation}}

\newcommand {\equ}[1]     {\eqref{eq:#1}}

\begin{document}

\title{Non-planarity and metric Diophantine approximation for systems of linear forms}

\author{Victor Beresnevich\\ {\small\sc(York)} \and
Dmitry Kleinbock \\ {\small \sc(Brandeis)}
\and
Gregory Margulis\\ {\small\sc(Yale)}
}

\date{}

\maketitle

\vspace*{-5ex}

\begin{abstract}
In this paper we develop a general theory of metric Diophantine approximation for systems of linear forms. A new notion of `weak non-planarity' of manifolds and more generally measures on the space $\Mat{m}{n}$ of $m\times n$ matrices over $\R$ is introduced and studied. This notion generalises the one of non-planarity in $\R^n$ and is used to establish strong (Diophantine) extremality of manifolds and measures in $\Mat{m}{n}$.
Thus our results 
contribute to resolving a 
problem 
stated in \cite[\S 9.1]{Gorodnik-07:MR2261070} regarding the strong extremality of manifolds in $\Mat{m}{n}$. Beyond the above main theme of the paper, we also develop a corresponding theory of inhomogeneous and weighted Diophantine approximation. In particular, we extend the recent inhomogeneous transference results of the first named author and Velani \cite{Beresnevich-Velani-10:MR2734962}
and use them to bring the inhomogeneous theory in balance with its homogeneous counterpart.
\end{abstract}

{\small \noindent\emph{Key words and phrases}: metric simultaneous Diophantine approximation, linear forms, strongly extremal manifolds, multiplicatively very well approximable points

\smallskip

 }



\section{Introduction}\label{intro}

Throughout $\Mat{m}{n}$ denotes the set of $m\times n$
matrices over $\R$ and $\|\vv \cdot\|$ stands for a norm on $\R^k$ which, without loss of generality, will be taken to be Euclidean. Thus $\|\vv x\| =\sqrt{x_1^2 + \ldots + x_k^2}$ for a $k$-tuple $\vv x=(x_1,\dots,x_k)\in\R^k$. We also define the following two functions of $\vv x=(x_1,\dots,x_k)\in\R^k$ that are particularly convenient for introducing the multiplicative form of Diophantine approximation:
$$
\Pi(\vv x)=\prod_{i=1}^k|x_i|\qquad\text{and}\qquad \Pi_+(\vv
x)=\prod_{i=1}^k\max\{1,|x_i|\}.
$$
We begin by recalling some fundamental concepts from the theory of Diophantine approximation. Let $Y\in\Mat{m}{n}$. If there exists $\ve>0$ such that the inequality
\begin{equation}\label{e:002}
\|Y\vv q-\vv p\|^m<\|\vv q\|^{-(1+\ve)n}
\end{equation}
holds for infinitely many $\vv q\in\Z^n$ and $\vv p\in\Z^m$, where $\vv q$ is regarded
as a column, then $Y$ is called \emph{very well approximable (VWA)}.
Further, if there exists $\ve>0$ such that the inequality
\begin{equation}\label{e:004}
\Pi(Y\vv q-\vv p)<\Pi_+(\vv q)^{-1-\ve}
\end{equation}
holds for infinitely many $\vv q\in\Z^n$ and $\vv p\in\Z^m$ then $Y$ is called \emph{very well multiplicatively approximable (VWMA)}. 
See Lemma~\ref{lemV} for an equivalent (and new) characterization of this property within a more general inhomogeneous setting.

One  says that a measure $\mu$ on $\Mat{m}{n}$ is  \emph{extremal}  (resp., \emph{strongly extremal}) if
$\mu$-almost all $Y\in\Mat{m}{n}$ are not VWA (resp., not VWMA).
It will be convenient to say that $Y$ itself is (strongly) extremal if so is the atomic measure supported at $Y$; in other words, if $Y$ is not very well (multiplicatively) approximable.

 It  is easily seen that
\begin{equation}\label{e:003}
\Pi_+(\vv q)\le \|\vv q\|^n\qquad\text{and}\qquad \Pi(Y\vv
q-\vv p)\le \|Y\vv q-\vv p\|^m
\end{equation}
for any $\vv q\in\Z^n\bnz$ and $\vv p\in\Z^m$. Therefore, (\ref{e:002}) implies (\ref{e:004}) and thus  strong extremality implies extremality. It is worth mentioning that if $\ve=0$ then (\ref{e:002}) as well as (\ref{e:004}) holds for infinitely many $\vv q\in\Z^n$ and $\vv p\in\Z^m$. The latter fact showing the optimality of exponents in (\ref{e:002}) and (\ref{e:004}) is due to Minkowski's theorem on linear forms -- see, \eg, \cite{Schmidt-1980}.

\medskip

The property of being strongly extremal is generic in $\Mat{m}{n}$.
Indeed, it is a relatively easy consequence of the Borel-Cantelli lemma that
Lebesgue measure on $\Mat{m}{n}$ is strongly extremal. However, when the entries of $Y$ are restricted by some functional relations (in other words $Y$ lies on a submanifold of $\Mat{m}{n}$)
investigating the corresponding measure  for extremality or strong extremality becomes much harder.
The study of manifolds for extremality goes back to the problem of Mahler \cite{Mahler-1932b} that
almost all points on the
Veronese curves $\{(x,\dots,x^n)\}$ (viewed as either row or column matrices)
are extremal. The problem was studied in depth for over 30
years and eventually settled by Sprind\v zuk in 1965 -- see
\cite{Sprindzuk-1969-Mahler-problem} for a full account. The far
more delicate conjecture that the Veronese curves in $\R^n$ are
\emph{strongly extremal}\/ (that is almost all points on the curves
are not VWMA) has been stated by Baker
\cite{Baker-75:MR0422171} and generalized by Sprind\v zuk \cite{Sprindzuk-1980-Achievements}.

It will be convenient to introduce the following definition (cf.\ \cite[\S 4]{Kleinbock-Tomanov-07:MR2314053}): say that  a subset $\cM$ of $\R^n$ is
\emph{non-planar}\/ if whenever $U$ is an open subset of $ \R^n$ containing at least one point of $\cM$, the intersection $\cM\cap U$ is not entirely contained in any
affine hyperplane of $\R^n$. Clearly the curve parametrized by $(x,\dots,x^n)$ is non-planar; more generally, if  $\cM$ is immersed into
$\R^n$ by an analytic map $\vv f=(f_1,\dots,f_n)$, then the
non-planarity of $\cM$ exactly means that the functions
$1,f_1,\dots,f_n$ are linearly independent over $\R$. Sprind\v zuk conjectured in 1980 that non-planar analytic submanifolds of $\R^n$ are strongly extremal.
There has been a sequence of partial results regarding the Baker-Sprind\v zuk problem but the complete solution was given in  \cite{Kleinbock-Margulis-98:MR1652916}.
In fact, a more general result was established there: strong extremality of smooth non-degenerate submanifolds.  Namely,  a submanifold $\cM$ 
is said to be \emph{non-degenerate}\/ if for almost every (with respect to the volume measure) point $\vv x$ of  $\cM$ one has
\begin{equation}\label{nd}
\R^n =
T_{\vv x}^{(k)}\cM\quad\text{for some }k\,,
\end{equation}
where $T_{\vv x}^{(k)}\cM$ is the $k$-th order tangent space to $\cM$ at ${\vv x}$ (the span of partial derivatives of a parameterizing map of orders up to $k$). 
It is not hard to see that any non-degenerate submanifold is non-planar while any non-planar analytic submanifold is non-degenerate. (In a way, non-degeneracy is an infinitesimal analog of the notion of non-planarity.)

The paper  \cite{Kleinbock-Margulis-98:MR1652916} also opened up the new avenues for investigating submanifolds of $\Mat{m}{n}$ for extremality and strong extremality. The following explicit problem was subsequently stated by Gorodnik as Question~35 in \cite{Gorodnik-07:MR2261070}:

\bigskip

\noindent\textbf{Problem 1: } \emph{Find reasonable and checkable conditions for a smooth submanifold $\cM$ of $\Mat{m}{n}$ which generalize non-degeneracy of vector-valued maps and impliy that almost every point of $\cM$ is extremal (strongly extremal).}

\bigskip

One can also pose a problem of generalizing the notion of non-planarity of subsets of $\R^n$ to those of $\Mat{m}{n}$, so that, when $\cM$ is an analytic submanifold, its non-planarity implies that almost every point of $\cM$ is extremal (strongly extremal). It is easy to see, e.g.\ from examples considered in \cite{Kleinbock-Margulis-Wang-09}, that
 being locally  not  contained in  proper affine subspaces  of $\Mat{m}{n}$ is not the right condition to consider.

\ignore{Recall that a connected analytic submanifold $\cM$ of $\R^n$ is
\emph{non-planar}\/ if $\cM$ is not entirely contained in a
proper hyperplane of $\R^n$ -- see e.g.\
\cite{Kleinbock-Margulis-98:MR1652916}. If $\cM$ is immersed into
$\R^n$ by an analytic map $\vv f=(f_1,\dots,f_n)$ then the
non-planarity of $\cM$ exactly means that the functions
$1,f_1,\dots,f_n$ are linearly independent over $\R$.}
Until recently the only examples of extremal manifolds of $\Mat{m}{n}$ with $\min\{m,n\}\ge2$ have been those found by Kovalevskaya \cite{Kovalevskaya-87:MR886162,
Kovalevskaya-87:MR904156}. She has considered submanifolds $\cM$
of $\Mat{m}{n}$ of dimension $m$ immersed by the map
\begin{equation}\label{e:005}
(x_1,\dots,x_m)\ \mapsto \ \left(\begin{array}{ccc}
f_{1,1}(x_1) & \dots & f_{1,n}(x_1) \\
\vdots & \ddots & \vdots \\
f_{m,1}(x_m) & \dots & f_{m,n}(x_m)
\end{array}
\right),
\end{equation}
where 
$f_{i,j}:I_i\to\R$ are $C^{n+1}$ functions defined on some intervals $I_i\subset \R$ such that every row in (\ref{e:005}) represents a non-degenerate map. Assuming that $m\ge n(n-1)$
Kovalevskaya has shown that $\cM$ is extremal. In the case $n=2$ and
$m\ge 2$ Kovalevskaya \cite{Kovalevskaya-87:MR934194} has also
established a stronger statement, which treats the inequality
$\|Y\vv q-\vv p\|^m<\Pi_+(\vv q)^{-(1+\ve)n}$ -- a mixture of (\ref{e:002}) and (\ref{e:004}).

In principle, manifolds (\ref{e:005}) are natural to consider but within the above results the dimensions $m$ and $n$ are bizarrely confined. The overdue general result regarding Kovalevskaya-type manifolds has
been recently established in \cite{Kleinbock-Margulis-Wang-09}. More precisely, it has been shown that any manifold of the form (\ref{e:005}) is strongly extremal provided that every row $(f_{i,1},\dots,f_{i,n})$ in (\ref{e:005}) is a non-degenerate map into $\R^n$ defined on an open subset of $\R^{d_i}$.

Working towards the solution of Problem~1 the following more general result has been established in \cite{Kleinbock-Margulis-Wang-09}. Let $\dd$ be the map defined on $\Mat{m}{n}$ that, to a given $Y\in\Mat{m}{n}$, assigns the collection of all minors of $Y$ in a certain fixed order. Thus $\dd$ is a map from $\Mat{m}{n}$ to $\R^N$, where $N=\binom{m+n}{n}-1$ is the number of all possible minors of an $m\times n$ matrix.
According to   \cite[Theorem~2.1]{Kleinbock-Margulis-Wang-09} any smooth submanifold $\cM$ of $\Mat{m}{n}$ such that $\dd(\cM)$ is non-degenerate is strongly extremal. The result also treats pushforwards of Federer measures -- see Theorem~\ref{kmw} 
for further details.


In the present paper we introduce a weaker (than in  \cite{Kleinbock-Margulis-Wang-09}) version of non-planarity of a subset of $\Mat{m}{n}$ which naturally extends the one for subsets of vector spaces and, in the smooth manifold case, is implied by the non-degeneracy of  $\dd(\cM)$. Then we use results of  \cite{Kleinbock-Margulis-Wang-09} to conclude (Corollary~\ref{t1}) that weakly non-planar analytic submanifolds of $\Mat{m}{n}$ are strongly extremal. See Theorem~\ref{t2} for a more general statement. 
The structure of the paper is as follows: we formally introduce the weak non-planarity condition and state our main results in \S\ref{main_results}. In the next section we compare our new condition with the one introduced in  \cite{Kleinbock-Margulis-Wang-09}. The 
main theorem is proved in \S\ref{proof}, while \S\ref{more} is devoted to some further features of the concept of weak non-planarity;  \S\ref{inhomsec} discusses an inhomogeneous extension of our main results, and the last section contains several  concluding remarks and open questions.
\bigskip

\noindent{\it Acknowledgements.} The authors are grateful to the University of Bielefeld for providing a stimulative research environment during their visits supported by SFB701. We gratefully acknowledge the support
of the National Science Foundation through grants DMS-0801064, DMS-0801195 and DMS-1101320, 
and of EPSRC through grants EP/C54076X/1 and EP/J018260/1.

\section{Main results}\label{main_results}


Let us begin by introducing some terminology and stating some earlier results. Let $X$ be a Euclidean space. Given $x\in X$ and $r>0$, let $B(x,r)$
denote the open ball of radius $r$ centred at $x$. If $V=B(x,r)$ and $c>0$, let $cV$ stand for $B(x,cr)$. Let $\mu$ be a measure on $X$. All the measures within this paper will be assumed to be Radon. Given $V\subset X$ such that $\mu(V)>0$ and a function $f:V\to\R$, let
$$
\|f\|_{\mu,V}=\sup_{x\in V\,\cap\,\supp\,\mu}|f(x)|.
$$
A Radon measure $\mu$ will be called \emph{$D$-Federer on $U$}, where $D>0$ and $U$ is an open subset of $X$, if $\mu(3V)<D\mu(V)$ for any ball $V\subset U$ centred in the support of $\mu$. The measure $\mu$ is called \emph{Federer} if for $\mu$-almost every point $x\in X$ there is a neighborhood $U$ of $x$ and $D>0$ such that $\mu$ is $D$-Federer on $U$.

Given $C,\alpha>0$ and an open subset $U\subset X$, we say that $f:U\to\R$ is \emph{$(C,\alpha)$-good on $U$ with respect to the measure $\mu$}\/ if for any ball $V\subset U$ centred in $\supp\,\mu$ and any $\ve>0$ one has
$$
\mu\big(\{x\in V:|f(x)|<\ve\}\big)\le C\left(\frac{\ve}{\|f\|_{\mu,V}}\right)^\alpha\mu(V)\,.
$$
Given $\vv f=(f_1,\dots,f_N):U\to\R^N$, we say that the pair
$(\vv f,\mu)$ is \emph{good}\/ if for $\mu$-almost every $x\in U$ there is a neighborhood $V\subset U$ of $x$ and  $C,\alpha > 0$ such that any linear combination of $1,f_1,\dots,f_N$ over $\R$ is $(C,\alpha)$-good on $V$.
The pair $(\vv f,\mu)$ is called \emph{non-planar}\/ if
\eq{nonpl}{\begin{aligned}\text{for any ball $V\subset U$ centered in $\supp\,\mu$},\qquad\qquad\qquad\\
\text{the set $\vv f(V\cap\supp\,\mu)$ is not contained in any affine hyperplane of $\R^N$}.\end{aligned}} Clearly it generalizes the definition of non-planarity given in the introduction: $\supp\,\mu$ is non-planar iff so is the pair $(\Id, \mu)$.


Basic examples of good and nonplanar pairs $(\vv f,\mu)$ 
are given by  $\mu = \lambda$ (Lebesgue
measure on $\R^d$) and $\vv f$ smooth and nondegenerate, see  \cite[Proposition 3.4]{Kleinbock-Margulis-98:MR1652916}.
The paper \cite{Kleinbock-Lindenstrauss-Weiss-04:MR2134453} introduces a class of {\sl friendly\/} measures:
  a measure $\mu$ on $\R^n$ is
friendly if and only if it is Federer and the pair $(\Id,\mu)$
is good and nonplanar.
In the latter paper the approach to metric \da\ developed in \cite{Kleinbock-Margulis-98:MR1652916}
has been extended to maps and measures satisfying the conditions described above.
One of its main results is the following statement, implicitly contained in  \cite{Kleinbock-Lindenstrauss-Weiss-04:MR2134453}:

\begin{theorem}\label{klw}{\rm \cite[Theorem~4.2]{dima pamq}}
Let $\mu$ be a Federer measure on
$\R^d$, $U\subset\R^d$ open, and $\vv f:U\to\R^n$ a continuous map such that $(\vv f,\mu)$ is
good and
nonplanar; then $\vv f_*\mu$
is strongly extremal.
\end{theorem}

Here and hereafter $\vv f_*\mu$ is the {\sl pushforward\/} of $\mu$ by $\vv f$, defined by $\vv f_*\mu(\cdot) \df \mu\big( \vv f^{-1}(\cdot)\big)$. When $\mu$ is Lebesgue measure and $\vv f$ is smooth and nonsingular,  $\vv f_*\mu$ is simply (up to equivalence) the volume measure on the manifold $\vv f(U)$.

\medskip

The next development came in the paper by Kleinbock, Margulis and Wang in 2011.
Given $F:U\to\Mat{m}{n}$, let us say that $(F,\mu)$ is \emph{good} if $(\dd\circ F,\mu)$ is good, where $\dd$ is the imbedding of $\Mat{m}{n}$ to $\R^N$ defined in \S\ref{intro}, where $N=\binom{m+n}{n}-1$.
Also we will say that $(F,\mu)$ is \emph{strongly non-planar} if \eq{strnonpl}{(\dd\circ F,\mu)\text{ is non-planar. }}
Clearly $\vv d$ is the identity map when $\min\{n,m\}=1$, thus in both  row-matrix and column-matrix cases \equ{strnonpl} is equivalent to \equ{nonpl}.
Therefore the following general result, established in \cite{Kleinbock-Margulis-Wang-09}, generalizes the above theorem:

\begin{theorem}\label{kmw}{\rm \cite[Theorem~2.1]{Kleinbock-Margulis-Wang-09}}
Let $U$ be an open subset of\/ $\R^d$, $\mu$ be a Federer measure on $U$ and $F:U\to\Mat{m}{n}$ be a continuous map such that
$(F,\mu)$ is {\rm(i)} good, and {\rm(ii)} strongly non-planar.  Then
$F_*\mu$ is strongly extremal.
\end{theorem}


In this paper we introduce a broader class of strongly extremal measures on $\Mat{m}{n}$ by relaxing condition (ii) of Theorem~\ref{kmw}. To introduce a weaker notion of non-planarity, we need the following notation:
given \eq{defab}{A\in \Mat{n}{m}(\R)\text{ and }B\in \Mat{n}{n}(\R)\,,}
 define
\begin{equation}\label{e:006}
\cH_{A,B} \df \{Y\in\Mat{m}{n}:\det(AY+B)=0\}.
\end{equation}

These sets will play the role of proper affine subspaces of vector spaces. It will be convenient to introduce notation
$\mathbf{H}_{m,n}$ for the collection of all sets $\cH_{A,B}\subset\Mat{m}{n}$ where $A\in\Mat{n}{m}$, $B\in\Mat{n}{n}$ and $\rank(A|B)=n$. Then
 for $F$ and $\mu$ as above, let us  say that $(F,\mu)$ is \emph{weakly non-planar} if
\eq{nonplgeneral}{
F(V\cap\supp\,\mu)\not\subset \cH\text{ for any ball $V\subset U$ centered in $\supp\,\mu$ and any }\cH\in \mathbf{H}_{m,n}\,.
}

Obviously $\cH_{A,B}=\varnothing$ if $A=0$. Otherwise, $\det(AY+B)$ is a non-constant polynomial and $\cH_{A,B}$ is a hypersurface in $\Mat{m}{n}$. Thus, the weak non-planarity of $(F,\mu)$ simply requires that $F(\supp\,\mu)$ does not locally lie entirely inside such a hypersurface. We shall see in the next section  that in both  row-matrix and column-matrix cases the weak non-planarity defined above is again equivalent to \equ{nonpl} (hence to strong non-planarity), and that in general strong non-planarity implies weak non-planarity but not vice versa. Thus the following theorem is a nontrivial generalization of Theorem~\ref{kmw}:

\begin{theorem}[Main Theorem]\label{t2} 
Let $U$ be an open subset of\/ $\R^d$, $\mu$ a Federer measure on $U$ and $F:U\to\Mat{m}{n}$ a continuous map such that
$(F,\mu)$ is {\rm(i)} good, and {\rm(ii)} weakly  non-planar. Then
$F_*\mu$ is strongly extremal.
\end{theorem}

Specializing to the case of submanifolds of $\Mat{m}{n}$, we can call
a smooth submanifold $\cM$ of $\Mat{m}{n}$
\emph{weakly non-planar}\/ if
\eq{mfldnonpl}{
V\cap \cM \not\subset \cH
\text{ for any ball $V$ centered in $\cM$ and any }\cH\in\mathbf{H}_{m,n}\,.
}
Then Theorem~\ref{t2} readily implies

\begin{corollary}\label{t1}
Any analytic  weakly non-planar submanifold of $\Mat{m}{n}$ is strongly
extremal.
\end{corollary}

\begin{proof} Without loss of generality we can assume that $\cM$ is immersed in $\Mat{m}{n}$ by an analytic map $F$ defined on  $\R^d$. Let $\mu$ be the  $d$-dimensional Lebesgue measure; then, saying that $\cM$ is strongly extremal is the same as saying that $F_*\mu$ is strongly extremal. To see that  $(F,\mu)$ is  weakly non-planar in the sense of \equ{nonpl} provided that $\cM$ is weakly non-planar in the sense of \equ{mfldnonpl}, take a ball $V\subset \R^d$ and assume  that $F(V) = F(V\cap\supp\,\mu)\subset \cH_{A,B}$  for some choice of $A\in  \Mat{n}{m}$ and $B\in  \Mat{n}{n}$ with $\rank(A|B)=n$. Clearly there exists a ball $U$ in $\Mat{m}{n}$ centered in $\cM$ and a ball $V'\subset V$ such that $U\cap \cM\subset F(V')$, contradicting to  \equ{mfldnonpl}.
Finally, the fact that $(F,\mu)$ is good is due to the analyticity of $F$ -- see \cite{Kleinbock-Margulis-Wang-09} or indeed \cite{Kleinbock-03:MR1982150}.
\end{proof}

\bigskip
We remark that if $\cM$  is a connected  analytic  submanifold of $\Mat{m}{n}$, then \equ{mfldnonpl} is simply equivalent to  $\cM$ not being contained in $ \cH$
for any $\cH\in\mathbf{H}_{m,n}$.

\bigskip

We postpone the proof of Theorem~\ref{t2} 
until \S\ref{proof}, after we compare the two (strong and weak) nonplanarity conditions introduced above.




\section{Weak vs.\ strong non-planarity}\label{vs}

Throughout this section $F:U\to\Mat{m}{n}$ denotes a map from an open subset $U$ of a Euclidean space $X$, and $\mu$ is a measure on $X$.

\medskip
The first result of the section shows that Theorem~\ref{kmw} is a consequence of Theorem~\ref{t2}:

\begin{lemma}\label{l:01}
If  $(F,\mu)$ is strongly non-planar, then it is weakly non-planar.
\end{lemma}

\begin{proof}
Let $(\dd\circ F,\mu)$ be non-planar. Let $A\in  \Mat{n}{m}$ and $B\in  \Mat{n}{n}$ with $\rank(A|B)=n$ and let $V\subset U$ be a ball centered in $\supp\,\mu$.
Observe that for any $Y\in \Mat{m}{n}$
$$
\left(\begin{array}{cc}
   I_m & 0 \\
   A & I_n
       \end{array}
\right)
\left(\begin{array}{cc}
   I_m & Y \\
   A & B
       \end{array}
\right)=
\left(\begin{array}{cc}
   I_m & Y \\
   0 & AY+B
       \end{array}
\right).
$$
Therefore,
\begin{equation}\label{e:044}
\det(AY+B)=\det
\left(\begin{array}{cc}
   I_m & Y \\
   A & B
       \end{array}
\right).
\end{equation}
By the Laplace identity, the right hand side of (\ref{e:044})
is a linear combination of minors of $Y$ and $1$ with the coefficients
being minors of order $n$ of $(A|B)$ taken with appropriate
signs. Since $\rank(A|B)=n$, these coefficients 
are not all
zero, therefore vanishing of (\ref{e:044}) defines either an affine hyperplane or the empty set.
Since $(\dd\circ F,\mu)$ is non-planar, it follows that (\ref{e:044})
does not vanish on $V\cap\supp\,\mu$, thus implying $F(V\cap\supp\,\mu)\not\subset \cH_{A,B}$. This verifies that $(F,\mu)$ is weakly non-planar and completes the proof.
\end{proof}

\bigskip

The converse to the above lemma is in general not true; here is a counterexample:

\bigskip

\begin{proposition}\label{2x2}
Let \begin{equation}\label{e:045new}Y = F(x,y,z) = \left(\begin{array}{cc}
         x & y \\
         z & x
       \end{array}
\right)\,,\end{equation} and let $\mu$ be Lebesgue measure on $\R^3$. Then $(F,\mu)$ is weakly but not strongly non-planar.
\end{proposition}


\begin{proof}
The fact that $\cM = F(\R^3)$ is not strongly non-planar is trivial because there are two identical minors (elements) in every $Y\in\cM$.
Now let $A,B\in \Mat{2}{2}$ with $\rank(A|B)=2$. By \equ{mfldnonpl} and in view of the analyticity of $F$, it suffices to verify that
\begin{equation}\label{e:046new}
\det(AY+B)\not=0\qquad\text{for some $Y$ of the form (\ref{e:045new}).}
\end{equation}
If $\det B\not=0$ then taking $Y=0$ proves (\ref{e:046new}). Also if
$\det A\not=0$ then ensuring (\ref{e:046new}) is very easy. Indeed, take
$Y$ of the form (\ref{e:045new}) with $y=z=0$ and $x$ sufficiently large.
Then $$\det(AY+B)=\det(xA+B)=x\det A\det(I_n+\tfrac1x BA^{-1})\not=0 \iff
\det(I_n+\tfrac1x BA^{-1})\not=0\,.$$ The latter condition is
easily met for sufficiently large $x$ because $\frac1x BA^{-1}\to0$
as $x\to\infty$. Thus for the rest of the proof we can assume that
$\det A=\det B=0$. Then without loss of generality we can also
assume that
$$
A=\left(\begin{array}{cc}
         \alpha_1 & \alpha_2 \\
         0 & 0
       \end{array}
\right)\qquad\text{and}\qquad B=\left(\begin{array}{cc}
         0 & 0 \\
         \beta_1 & \beta_2
       \end{array}
\right),
$$
otherwise we can use Gaussian elimination method to replace $A$ and
$B$ with the matrices of the above form. For $\rank(A|B)=2$ we have
that at least one of $\alpha_1$ and $\alpha_2$ is non-zero and at
least one of $\beta_1$ and $\beta_2$ is non-zero. For $Y$ is of the form (\ref{e:045new}), we have
$$
AY+B=\left(\begin{array}{cc}
         \alpha_1x+\alpha_2z & \alpha_1y+\alpha_2x \\
         \beta_1 & \beta_2
       \end{array}
\right).
$$
If $\alpha_1\not=0$ and $\beta_1\not=0$ then taking $x=0$, $y=1$,
$z=0$ ensures (\ref{e:046new}).\\
If $\alpha_2\not=0$ and $\beta_1\not=0$ while $\alpha_1=0$ then
taking $x=1$ and
$z=0$ ensures (\ref{e:046new}).\\
If $\alpha_2\not=0$ and $\beta_2\not=0$ then taking $x=0$, $y=0$ and
$z=1$ ensures (\ref{e:046new}).\\
If $\alpha_1\not=0$ and $\beta_2\not=0$ while $\alpha_2=0$ then
taking $x=1$ and $y=0$ ensures (\ref{e:046new}).
\end{proof}

\bigskip

We remark that  the non-planarity of $\vv d(\cM)$ for $\cM$ as above fails over $\Z$, and still $\cM$ is strongly extremal in view of Theorem~\ref{t2}.

\bigskip


Note however that in the case when matrices are rows/columns,
conditions \equ{nonpl} and \equ{strnonpl} are equivalent. This readily follows from 

\begin{lemma}\label{l:02}
Let $\min\{n,m\}=1$. Then for any $A\in \Mat{n}{m}$ and $B\in
 \Mat{n}{n}$ such that $\rank(A|B)=n$, the equation $\det(AY+B)=0$
defines either a hyperplane or an empty set.
\end{lemma}

\begin{proof}
First consider the case $n=1$. Then
$A=(a_1,\dots,a_m)\in  \Mat{1}{m}$, $B=(b)\in  \Mat{1}{1}$ and
$Y=(y_1,\dots,y_m)^t\in  \Mat{1}{1}$. Obviously, $AY+B=0$ becomes
$\sum_{i=1}^ma_iy_i+b=0$. Since $\rank(A|B)=1$, one of the coefficients is non-zero, and the claim follows.

\medskip

Consider now the case $m=1$. Then $A=(a_1,\dots,a_n)^t\in  \Mat{1}{n}$,
$B\in  \Mat{n}{n}$ and $Y=(y_1,\dots,y_n)\in  \Mat{1}{n}$.
By (\ref{e:044}),
\begin{equation}\label{e:007}
\det(AY+B)=0\qquad\Longleftrightarrow\qquad \det \left(\begin{array}{cc}
   1 & Y \\
   A & B
       \end{array}
\right)=\det B+\sum_{i=1}^na_iy_i=0,
\end{equation}
where $a_i$ is the cofactor of $y_i$. Since $\rank(A|B)=n$, as least
one of the numbers $\det B,a_1,\dots,a_n$ is non-zero. If
$a_1=\dots=a_n=0$ then $\det B\not=0$ and (\ref{e:007}) defines an
empty set. Otherwise, (\ref{e:007}) obviously defines a hyperplane.
\end{proof}

\section{Proof of Theorem~\ref{t2} 
}\label{proof}

Let us first express 
subsets $\cH_{A,B}$ of $\Mat{m}{n}$ in several equivalent ways. It will be convenient to introduce the following notation: we let   $W = \R^{m+n}$,
denote by $\vv e_1,\dots,\vv e_{m+n}$ the standard basis of $W$, and, for $i = 1,\dots,m+n$,  by $E^+_i$ (resp., $E^-_i$) the span of  the first (resp., the last) $i$ vectors of this basis, and by $\pi^+_i$ (resp., $\pi^-_i$) the orthogonal projection of $W$ onto $E^+_i$ (resp., $E^-_i$). Also, if $I =
\{i_1,\dots,i_{\ell}\}\subset \{1,\dots,m+n\}$ (written in the increasing order), we denote $\vv e_{I} \df
\vv e_{i_1}\wedge\dots\wedge \vv e_{i_{\ell}}\in \bigwedge^\ell(W)\,.$
For $A,B$ as in \equ{defab}, let $W_{A,B}$ be the subspace of $W$ spanned by the columns of the matrix
$\begin{pmatrix}A^t \\ B^t\end{pmatrix} = (A | B)^t$.  (Here and hereafter the superscript $t$ stands for transposition.) Note that
$\dim(W_{A,B}) = n$ due to the assumption on the rank of $(A | B)$.

Given $Y\in\Mat{m}{n}$, let us denote
\eq{defuy}{
u_Y \df
\begin{pmatrix}
   I_m & Y \\
    0 & I_n
       \end{pmatrix}.
}
Then we have the following elementary

\begin{lemma}\label{elem} The following are equivalent:
\begin{itemize}
\item[\rm (i)] $Y\in \cH_{A,B}$;
\item[\rm (ii)] $\dim\big( \pi^-_n(u_Y^t W_{A,B})\big)< n$;
\item[\rm (iii)] $u_Y^t W_{A,B}\cap E^+_m \ne \{0\}$;
\item[\rm (iv)] $(u_Y^t W_{A,B})^\perp\cap E^-_n \ne \{0\}$;
\item[\rm (v)]
$\dim \Big(\pi^+_m\left((u_Y^t W_{A,B})^\perp\right)\Big) < m$;
\item[\rm (vi)] $Y^t\in \cH_{D,-C}$, where $C\in  \Mat{m}{m}$ and $D\in  \Mat{m}{n}$ are such that
the columns of $
(C | D)^t$ form a basis for $W_{A,B}^\perp$.
\end{itemize}

\end{lemma}

\begin{proof}
Note that  $u_Y^t W_{A,B}$ is spanned by the columns of the matrix
$$
u_Y^t\begin{pmatrix}A^t \\ B^t\end{pmatrix} =
\begin{pmatrix}
   I_m & 0 \\
   Y^t & I_n
       \end{pmatrix}\begin{pmatrix}A^t \\ B^t\end{pmatrix}  = \begin{pmatrix}A^t \\ (AY + B)^t \end{pmatrix},
       $$
and $\pi^-_n(u_Y^t W_{A,B})$
is therefore spanned by the columns of $(AY + B)^t$. Since the latter matrix has rank less than $n$ if and only if (i) holds, the equivalence between (i) and (ii) follows. The equivalence (ii) $\iff$ (iii), together with  (iv) $\iff$ (v),  is a simple exercise in linear algebra. To derive  (iii) $\iff$ (iv), observe that dimensions of $(u_Y^t W_{A,B})^\perp$ and $  E^-_n$   add up to $\dim(W)$, therefore these two subspaces  have trivial
intersection if and only if the same is true for their orthogonal
complements.

Finally, to establish (v) $\iff$ (vi), it suffices to note that $$(u_Y^t W_{A,B})^\perp = \big((u_Y^t)^t\big)^{-1} W_{A,B}^\perp = u_{-Y} W_{A,B}^\perp$$ is spanned by the (linearly independent) columns of the matrix
$$
u_{-Y}\begin{pmatrix}C^t \\ D^t\end{pmatrix} =
\begin{pmatrix}
   I_m & -Y \\
   0 & I_n
       \end{pmatrix}\begin{pmatrix}C^t \\ D^t\end{pmatrix}  = \begin{pmatrix}C^t - YD^t \\ D^t \end{pmatrix}
       $$
and its orthogonal projection  onto $E^+_m$ is therefore spanned by the columns of $$C^t - YD^t = -(DY^t - C)^t\,.$$
Hence (v) holds if and only if $\det(DY^t - C) = 0$.
\end{proof}

\bigskip

\ignore{We will use the above lemma to establish the invariance of  weak non-planarity under transposition. For this we will need to consider the space $W^*$ dual to $W$, which we will identify with the $(m+n-1)$-fold exterior power of $W$, and let $(\vv e_1^*,\dots,\vv e_{m+n}^*)$ be the basis of $W^*$ dual to $(\vv e_1,\dots,\vv e_{m+n})$ and $\langle\cdot,\cdot\rangle$ the canonical pairing with $\langle\vv e_i,\vv e_j^*\rangle = \delta_{ij}$. The action of $u_Y^t$ on $W$ induces a dual action on $W^*$ which with respect to the basis   $(\vv e_1^*,\dots,\vv e_{m+n}^*)$ is given by $u_Y$.}

Now let $F:U\to\Mat{m}{n}$ be a map from an open subset $U$ of a Euclidean space $X$, $\mu$ a measure on $X$, and denote by $F^t:U\to  \Mat{n}{m}$ the map given by $F^t(x)=\big(F(x)\big)^t$.


\begin{corollary}\label{l:05}
$(F,\mu)$ is weakly non-planar if and only if $(F^t,\mu)$ is weakly non-planar.
\end{corollary}

\begin{proof}
Suppose that $(F,\mu)$ is not weakly non-planar, that is \equ{nonplgeneral} does not hold. Then there exists a ball $V$ centred in $\supp\,\mu$ such that $F(V\cap\supp\,\mu)\subset \cH_{A,B}$ for some $A\in\Mat{n}{m}$ and $B\in\Mat{n}{n}$ with $\rank(A|B)=n$. Using the equivalence (i) $\iff$ (vi) of the previous lemma, we conclude that there exist $C\in \Mat{m}{m}$ and $D\in \Mat{m}{n}$ such that
$$\rank(-C|D)=\rank(C|D)=m\quad\text{and}\quad F^t(V\cap\supp\,\mu)\subset \cH_{D,-C}\,;
$$
 hence $F^t$ is not weakly non-planar. Converse is proved similarly.
\end{proof}

\bigskip

The main theorem
will be derived using the approach based on dynamics on the space of lattices, which was first developed by Kleinbock and Margulis in  \cite{Kleinbock-Margulis-98:MR1652916} and then extended in \cite{Kleinbock-Margulis-Wang-09}.
The key observation here is the fact that   \di\ properties of $Y\in\Mat{m}{n}$ can be expressed in terms of
 of the action of diagonal matrices in $\SL_{m+n}( \R)$ on $$
 u_Y\Z^{m+n} = \left\{\begin{pmatrix} Y\vq - \vp\\\vq\end{pmatrix} : \vp\in\Z^m,\  \vq\in\Z^n\right\}\,.
$$
The latter object is a lattice in $W$ which is viewed as a point of the homogeneous space $\SL_{m+n}( \R)/\SL_{m+n}( \Z)$ of unimodular lattices in $W$. However we are able to use the final outcome of the techniques developed in \cite{Kleinbock-Margulis-Wang-09} and preceding papers, thus in this paper there is no need to state the quantitative nondivergence estimates (\cite[Theorem~5.2]{Kleinbock-Margulis-98:MR1652916}, \cite[Theorem~4.3]{Kleinbock-Lindenstrauss-Weiss-04:MR2134453}) and the correspondence between \da\ and dynamics on the space of lattices  \cite[Proposition 3.1]{Kleinbock-Margulis-Wang-09}. The reader is referred to the aforementioned paper, as well as to survey papers \cite{dima pamq, dima clay} for more details.

\medskip

Now let us  introduce some more notation.
For an $(m+n)$-tuple $\vv t = (t_1,\dots,t_{m+n})$ of real numbers, 
define
 $$g_\vt \df \diag(e^{t_1}, \ldots,
e^{ t_m}, e^{-t_{m+1}}, \ldots,
e^{-t_{m+n}})
\,.$$
We will denote by $\fa
$
the set of $(m+n)$-tuples $\vt$
such that
\eq{sumequal}{
t_1,\dots,t_{m+n} > 0
\quad \mathrm{and}\quad
\sum_{i = 1}^m t_i =\sum_{j = 1}^{n} t_{m+j} \,.
}
For a fixed $\vt\in\fa$ let us denote by $\cE^+_\vt$ the span of all
the eigenvectors of $g_\vt$ in  $\bigwedge(W)$ with eigenvalues greater or equal to one (in other words, those
which are not contracted by the $g_\vt$-action).
It is easy to see that $\cE^+_\vt$ is spanned by
elements $\vv e_I \wedge \vv e_J$ where
$I \subset
\{1,\dots,m\}$ and $J \subset
\{m+1,\dots,m+n\}$ are such
that
\eq{condeplus}{\sum_{i\in I}t_{i} \ge \sum_{j \in J}t_{j}\,.}
We will let
 $\pi^+_\vt$ be the orthogonal projection onto $\cE^+_\vt$.

For
$1\le \ell \le m+n-1$,   let us  denote by $
\mathcal{W}^\ell$
the set of {\sl decomposable elements\/} of $\bigwedge^\ell(W)$ (that is, elements which can be written as $\vw = \vv v_1 \wedge\dots \wedge \vv v_\ell$, where $\vv v_i\in W$), and denote $
\mathcal{W} \df \bigcup_{\ell = 1 }^{m+n-1}\mathcal{W}^\ell$.  Up to a sign the nonzero elements of $
\mathcal{W}^\ell$ can be identified with subgroups of
$W$ of rank $\ell$.
\ignore{A special attention will be paid to  decomposable elements 
with integer coordinates: we will let
 $$
\mathcal{W}^\ell_\Z\df \mathcal{W}^\ell \cap  \textstyle\bigwedge(\Z^{m+n})\quad\text{and}\quad\mathcal{W}_\Z \df \bigcup_{\ell = 1 }^{m+n-1}\mathcal{W}^\ell_\Z\subset \textstyle\bigwedge (\Z^{m+n})\,.
$$}

The next statement is a simplified version of  Corollary 5.1 
from \cite{Kleinbock-Margulis-Wang-09}:

\begin{theorem}\label{thm: criterion} 
Let     an open subset  $U$  of $\R^d$, a  continuous
map
  $F: U\to\Mat{m}{n}
$
and 
 a  Federer
 measure   $\mu$ on
$U$ be given.
Suppose that 
$(F,\mu)$ is good, and also 
that  for any  ball
$V\subset U$ with $\mu(V) > 0$ there exists positive  $c
$ such that
 \eq{cor condition}{
 \|\pi^+_\vt u_{F(\cdot)}\vw\|_{\mu,V} \ge c \quad \text{ for all } \vw\in\fw_\Z\bnz \text{ and }\vt\in\fa
 \,.}
Then $F_*\mu$ is strongly extremal.
\ignore{\item[\rm (ii)] {\rm \cite[Corollary 5.2]{Kleinbock-Margulis-Wang-09}} On the other hand, suppose that there exists a subset $V\subset U$ with $\mu(V) > 0$ such that
 \eq{cor condition z}{
\forall\,x\in V\text{ one has }\pi^+_\vt u_{F(\cdot)}\vw = 0 \quad \text{ for some } \vw\in\fw_\Z\bnz \text{ and }\vt\in\fa
 \,.}
Then  $F_*\mu$ is not strongly extremal.
\end{itemize}}
\ignore{ if and only if
 for any  ball
$V\subset U$ with $\mu(V) > 0$ 
and any $\beta
> 0$  there exists $T  > 0$ such that 
and any $\vw\in\fw$ one has
\eq{condition}{\|g_\vt u_{F(\cdot)}\vw\|_{\mu,V} \geq e^{-\beta t}\quad\forall\,\vw\in\fw_\Z \text{ and any }
\,\vt\in\fa\text{ with }t\ge T\,.}}
\end{theorem}

\ignore{Observe that part (ii) requires no additional assumptions on $F$ and $\mu$. Note also that   \equ{cor condition} and \equ{cor condition z} do not exhaust all the possibilities; s}
See also  \cite[Theorem~4.3]{Kleinbock-Margulis-Wang-09} for a necessary and sufficient condition for strong extremality in the class of good pairs $(F,\mu)$.

\medskip

Now we can proceed with the proof of our main theorem. 
\medskip

\noindent{\it Proof of Theorem~\ref{t2}.}
For $F$ and $\mu$ as in Theorem~\ref{t2}, we need to take a ball
$V\subset U$ with $\mu(V) > 0$ (which we can without loss of generality center at a point of $\supp\,\mu$) and find $c> 0$ such that \equ{cor condition} holds.
Since  $(F,\mu)$ is  weakly non-planar, from
 the equivalence (i) $\iff$ (vi) of Lemma~\ref{elem} we conclude that for any  $C\in  \Mat{m}{m}$ and $D\in  \Mat{m}{n}$ with $\rank(D|$\,$-C)=m$ one has
$\det\big(DF(x)^t - C\big) \ne 0$ for some $x\in\supp\,\mu\cap V$.  Equivalently, for any   $\vw\in
\mathcal{W}^m\bnz$, which we take to be the exterior product of columns of $\begin{pmatrix}-C^t \\ D^t\end{pmatrix}$, the orthogonal projection of $u_{F(x)}\vw$ onto $\bigwedge^m(E_m^+)$, which is equal to the exterior product of columns of
$$\begin{pmatrix}I & F(x)\end{pmatrix}\begin{pmatrix}-C^t \\ D^t\end{pmatrix} = -C^t + F(x) D^t = \big(DF(x)^t - C\big)^t\,,$$
 is nonzero for some $x\in\supp\,\mu\cap V$.

 \medskip

 Our next goal is to treat $\vw\in
\mathcal{W}^\ell$ with $\ell\ne m$ in a similar way. For this, let us consider the subspace $\cE^+$ of $\bigwedge(W)$ defined by
 \eq{basisplus}{ \mathcal E^+ \df \Span\big\{\vv e_I, \vv e_{\{1,\dots,m\}}\wedge \vv e_J: I \subset
\{1,\dots,m\},J \subset
\{m+1,\dots,m+n\}\big\}\,,}
or, equivalently, by
$$
\mathcal E^+ \cap \textstyle\bigwedge^\ell(W) = \begin{cases} \textstyle\bigwedge^\ell(E^+_m)\qquad\qquad\qquad\ \text{ if } \ell \le m\\
\R\vv e_{\{1,\dots,m\}}\wedge  \textstyle\bigwedge^{\ell-m}(E^-_n)\text{ if } \ell \ge m\end{cases}$$
In particular, $\mathcal E^+ \cap \textstyle\bigwedge^m(W) = \bigwedge^m(E^+_m)$ is one-dimensional and is spanned by $\vv e_{\{1,\dots,m\}}$.

\bigskip
The relevance of the space $\cE^+$ to our set-up is highlighted by

\begin{lemma}\label{eplus}  $\cE^+ = \bigcap_{\vt\in\fa}\cE^+_\vt$. \end{lemma}

\begin{proof} The direction $\subset$  is clear from \equ{basisplus} and the validity of \equ{condeplus} when either $J= \varnothing$ or $I = \{1,\dots,m\}$. Conversely, take $\vw\in \bigwedge(W)$ and suppose that  there exist a proper subset $I $ of
$\{1,\dots,m\}$ and a nonempty subset  $J$
of $\{m+1,\dots,m+n\}$ such that the orthogonal projection of $\vw$ onto $\vv e_I \wedge \vv e_J$ is not zero. Then choose $\vt\in\fa\bnz$ such that $t_i = 0$ when $i\in I$, and $t_j\ne 0$ when $j\in J$; this way $\vv e_I \wedge \vv e_J$ is contracted by $g_\vt$, which implies that $\vw$ is not contained in $\cE^+_\vt$.
 \end{proof}

\bigskip

Denote by $\pi^+$ the orthogonal projection $\bigwedge(W)\to\cE^+$; thus we have shown that
\eq{positive}{\|\pi^+ u_{F(\cdot)}\vw\|_{\mu,V} > 0\quad\forall\,\vw\in
\mathcal{W}^m\bnz\,.}
We now claim that the same is true for all  $\vw\in
\mathcal{W}\bnz$. Indeed, take
 $$\vw = \vv v_1 \wedge\dots \wedge \vv v_\ell\ne 0\,,$$ where $\ell < m$, and choose arbitrary $\vv v_{\ell +1},\dots, \vv v_m$ such that $\vv v_{1},\dots, \vv v_m$ are linearly independent.
Then $\pi^+\big(u_{F(x)}(\vv v_1 \wedge\dots \wedge \vv v_m)\big)$ being nonzero is equivalent to $\pi^+(u_{F(x)}\vv v_{1}),\dots, \pi^+(u_{F(x)}\vv v_m)$ being linearly independent, which implies $\pi^+(u_{F(x)}\vv v_{1}),\dots, \pi^+(u_{F(x)}\vv v_\ell)$  being linearly independent, i.e.\ $\pi^+(u_{F(x)}\vw)\ne 0$.

The case  $\ell > m$ can be treated in a dual fashion: if $\vw = \vv v_1 \wedge\dots \wedge \vv v_\ell\ne 0$ is such that $\pi^+(u_{F(x)} \vw)=  0$, then there exists $\vv v\in E^+_m$ which is orthogonal to all of $\pi^+(u_{F(x)}\vv v_{1}),\dots, \pi^+(u_{F(x)}\vv v_\ell)$, hence to all of $\pi^+(u_{F(x)}\vv v_{1}),\dots, \pi^+(u_{F(x)}\vv v_m)$, and the latter amounts to saying that  $\pi^+\big(u_{F(x)}(\vv v_1 \wedge\dots \wedge \vv v_m)\big)=  0$, contradicting \equ{positive}.

\medskip
Notice that we have proved that for any ball $V\subset U$   centered in $\supp\,\mu$ , the (continuous) function $$\vw\mapsto \|\pi^+(u_{F(\cdot)}\vw)\|_{\mu,V}$$ is nonzero on the intersection of $\mathcal{W}$ with the unit sphere in $\bigwedge(W)$, hence, by compactness, it has a uniform lower bound. Since $\|\vw\|\ge 1$ for any $\vw\in\fw_\Z\bnz$, it follows that for any $V$ as above there exists $c > 0$ such that
$$
 \|\pi^+ u_{F(\cdot)}\vw\|_{\mu,V} \ge c \quad \text{ for all } \vw\in\fw_\Z\bnz\,.
$$
This, in view of Lemma~\ref{eplus}, finishes the proof of \equ{cor condition}.

\hskip 6in $\boxtimes$
\ignore{
\bigskip

\noindent{\it Proof of Theorem~\ref{t3}.} 
We are given that the weak non-planarity of $(F,\mu)$ fails over $\Z$; that is, there exists a ball  $V\subset U$ centered in $\supp\,\mu$  and $A\in\Mat{n}{m}(\Z)$, $B\in\Mat{n}{n}(\Z)$ with $\rank(A|B)=n$ such that  $F(V\cap\supp\,\mu)$ is  contained in  $\cH_{A,B}$. Using the equivalence (i) $\iff$ (vi) of Lemma~\ref{elem} again, we find $\vw\in\mathcal{W}^m_\Z$ (the exterior product of columns of the matrix $\begin{pmatrix}-C^t \\ D^t\end{pmatrix}$ whose entries can be chosen to be integers) such that $\pi^+ u_{F(x)}\vw = 0$ for all $x\in \supp\,\mu\cap V$. Consequently, for any $x\in V\cap\supp\,\mu$ one has
$$u_{F(x)}\vw\in(\cE^+)^\perp = \bigcup_{\vt\in \fa}(\cE^+_\vt)^\perp$$
(the last equality holds  in view of  Lemma~\ref{eplus}). Therefore  \equ{cor condition z} is satisfied with $U'$ equal to $ \supp\,\mu\cap V$.

\hskip 6in $\boxtimes$

}

\section{
More about weak non-planarity}\label{more}

The set of strongly extremal matrices in $\Mat{m}{n}$ is invariant under various natural transformations. For example, it is invariant under non-singular rational transformations, in particular, under the permutations of rows and columns,  and, in view of Khintchine's Transference Principle \cite{SW}, under transpositions. Also, if a matrix $Y\in\Mat{m}{n}$ is strongly extremal then any submatrix of $Y$ is strongly extremal. We have already shown in \S\ref{proof} that weak non-planarity is invariant under transposition; in this section we demonstrate some additional  invariance properties.

As before, throughout this section $F:U\to\Mat{m}{n}$ denotes a map from an open subset $U$ of a Euclidean space $X$ and $\mu$ is a measure on $X$. The following statement shows the invariance of weak non-planarity under non-singular transformations.

\begin{lemma}\label{l:03}
Assume that $(F,\mu)$ is weakly non-planar. Let
$L\in\GL_m(\R)$ and $R\in\GL_n(\R)$ be given and let
$\tilde F:U\to\Mat{m}{n}$ be a map given by $\tilde F(x)=LF(x)R$ for $x\in U$. Then $(\tilde F,\mu)$ is weakly non-planar.
\end{lemma}

\begin{proof}
Take any $\tilde A\in  \Mat{n}{m}$ and $\tilde B\in \Mat{n}{n}$
such that $\rank(\tilde A|\tilde B)=n$ and let $V\subset U$ be a ball centered in $\supp\,\mu$. Define $A=\tilde AL$ and
$B=\tilde BR^{-1}$. It is easily seen that
$$
(A|B)=(\tilde A|\tilde B)\left(\begin{array}{cc}
                    L & 0\\
                    0 & R^{-1}
                  \end{array}
\right),
$$
that is,  the product of $(\tilde A|\tilde B)$ by a  non-singular matrix; thus $\rank(A|B)=\rank(\tilde A|\tilde B)=n$. Since $(F,\mu)$ is weakly non-planar, $F(V\cap\supp\,\mu)\not\subset \cH_{A,B}$. Therefore, there exists $x\in V\cap\supp\,\mu$ such that $\det(AF(x)+B)\not=0$. 
Then
\begin{equation}\label{e:008}
\begin{array}[b]{rcl}
AF(x)+B & = & A(L^{-1}\tilde F(x)R^{-1})+B\\[1ex]
       & = & ((AL^{-1})\tilde F(x)+BR)R^{-1}\\[1ex]
       & = & (\tilde A\tilde F(x)+\tilde B)R^{-1}.
\end{array}
\end{equation}
Since $\det R\not=0$ and $\det(AF(x)+B)\not=0$, (\ref{e:008}) implies that $\det(\tilde A\tilde F(x)+\tilde B)\not=0$. This means that $\tilde F(V\cap\supp\,\mu)\not\subset \cH_{\tilde A,\tilde B}$. The proof is complete.
\end{proof}

\bigskip

Taking $L$ and $R$ to be $I_m$ and $I_n$ with permuted columns/rows readily implies (as a corollary of Lemma~\ref{l:03}) that \emph{weak non-planarity in invariant under permutations of rows and/or columns in $F$}. The next statement is a natural generalisation of Lemma~\ref{l:03}.

\begin{lemma}\label{l:04}
Assume that $(F,\mu)$ is weakly non-planar. Let $\tilde n\le n$, $\tilde m\le m$ and $L\in  \Mat{\tilde m}{m}$ and $R\in  \Mat{n}{\tilde n}$ and let $\tilde F:U\to  \Mat{\tilde m}{\tilde n}$ be a map given by $\tilde F(x)=LF(x)R$ for $x\in U$. If $\rank L=\tilde m$ and\/ $\rank R=\tilde n$ then $(\tilde F,\mu)$ is also weakly non-planar.
\end{lemma}

\begin{proof}
Since $\rank L=\tilde m$ and\/ $\rank R=\tilde n$, there are $C\in\GL_m(\R)$, $\tilde C\in\GL_{\tilde m}(\R)$, $D\in\GL_n(\R)$ and $\tilde D\in\GL_{\tilde n}(\R)$ such that $L=\tilde CL_0C$ and $R=DR_0\tilde D$, where $L_0=(I_{\tilde m}|0)$ and $R_0=(I_{\tilde n}|0)^t$. By Lemma~\ref{l:03}, $(F_1,\mu)$ is weakly non-planar, where $F_1(x)=CF(x)D$. Obviously, $\tilde F=\tilde CF_2(x)\tilde D$, where $F_2(x)=L_0F_1(x)R_0$. Therefore, by Lemma~\ref{l:03} again, the fact that $(\tilde F,\mu)$ is weakly non-planar would follow from the fact that $(F_2,\mu)$ is weakly non-planar. Thus, without loss of generality, within this proof we can simply assume that $L=L_0$ and $R=R_0$.

Take any $\tilde A\in \Mat{\tilde m}{\tilde n}$ and $\tilde B\in \Mat{\tilde n}{\tilde n}$ such that $\rank(\tilde A|\tilde B)=\tilde n$. Let
\begin{equation}\label{e:009}
A=\left(\begin{array}{ccc}
       \tilde A & 0 \\
       0 & 0
       \end{array}
\right)\in \Mat{n}{m}\qquad\text{and}\qquad
B=\left(\begin{array}{ccc}
       \tilde B & 0 \\
       0 & I_{n-\tilde n}
       \end{array}
\right)\in  \Mat{n}{n}.
\end{equation}
It is easily seen $\rank (A|B)=\rank(\tilde A|\tilde B)+n-\tilde n=n$. Take any ball $V$ centred in $\supp\,\mu$. Since $(F,\mu)$ is weakly non-planar, there is $x\in V\cap\supp\,\mu$ such that $\det(AF(x)+B)\not=0$. It is easily seen that $F(x)$ has the form
$$
F(x)=\left(\begin{array}{ccc}
       \tilde F(x) & * \\
       * & *
       \end{array}
\right),
$$
where $\tilde F(x)=L_0F(x)R_0\in \Mat{\tilde m}{\tilde n}$. Then using (\ref{e:009}) we get
$$
AF(x)+B=\left(\begin{array}{ccc}
       \tilde A\tilde F(x)+\tilde B & * \\
       0 & I_{n-\tilde n}
       \end{array}
\right).
$$
It follows that $\det(AF(x)+B)=\det(\tilde A\tilde F(x)+\tilde B)\not=0$, whence the claim of the lemma readily follows.
\end{proof}

\bigskip

Taking $L$ to be $L_0$ with permuted columns and $R$ to be $R_0$
with permuted rows readily implies (as a corollary of Lemma~\ref{l:04}) that \emph{any submatrix in a weakly non-planar $F$ is weakly non-planar}. Note that, combined with Proposition~\ref{2x2}, this shows that for any $m,n$ with $\min\{m,n\} > 1$ there exists a submanifold of $\Mat{m}{n}$ which is weakly but not strongly non-planar.

\bigskip

In the final part of this section we will talk about products of weakly non-planar measures.
In essence, strongly non-planar (and thus weakly non-planar) manifolds given by (\ref{e:005}) are products of non-planar rows. One can generalise this construction by considering products of matrices with arbitrary dimensions. For the rest of the section we will assume that $X_1$ and $X_2$ are two Euclidean spaces and $\mu_1$ and $\mu_2$ are Radon measures on $X_1$ and $X_2$ respectively.

\begin{lemma}\label{l:06}
For $i=1,2$ let $U_i$ be an open set is $X_i$ and let\/ $F_i:U_i\to \Mat{m_i}{n}(\R)$\/ be given. Let $\mu=\mu_1\times\mu_2$ be the product measure over $X=X_1\times X_2$ and let $F:U\to\Mat{m}{n}$, where $U=U_1\times U_2$ and $m=m_1+m_2$, be given by
\begin{equation}\label{prod}
F(x_1,x_2)\df \left(\begin{array}{c} F_1(x_1)\\ F_2(x_2)\end{array}\right).
\end{equation}
Assume that $(F_1,\mu_1)$ and $(F_2,\mu_2)$ are weakly non-planar. Then $(F,\mu)$ is weakly non-planar.
\end{lemma}

In view of Corollary~\ref{l:05} the following statement is equivalent to Lemma~\ref{l:06}.

\begin{lemma}\label{l:08}
For $i=1,2$ let $U_i$ be an open set is $X_i$ and let\/ $F_i:U_i\to \Mat{m}{n_i}(\R)$\/ be given. Let $\mu=\mu_1\times\mu_2$ be the product measure over $X=X_1\times X_2$ and let $F:U\to\Mat{m}{n}$, where $U=U_1\times U_2$ and $n=n_1+n_2$, be given by
$$
F(x_1,x_2)\df \Big(\ F_1(x_1)\ \ \ F_2(x_2)\ \Big).
$$
Assume that $(F_1,\mu_1)$ and $(F_2,\mu_2)$ are weakly non-planar. Then $(F,\mu)$ is weakly non-planar.
\end{lemma}

In order to prove Lemma~\ref{l:06} we will use the following
auxiliary statement.

\begin{lemma}\label{l:07}
Let $(F,\mu)$ be weakly non-planar, $r\le n$, $A\in  \Mat{r}{m}$, $B\in  \Mat{r}{n}$ and let $\rank(A|B)=r$. Then for any ball $V\subset U$ centred in $\supp\,\mu$ there is $x\in V\cap\supp\,\mu$ such that \/ $\rank(AF(x)+B)=r$.
\end{lemma}

\begin{proof}
Let $V\subset U$ be a ball centred in $\supp\,\mu$. Since $\rank(A|B)=r$, there are matrices $\tilde A\in
\Mat{n-r}{m}$ and $\tilde B\in \Mat{n-r}{n}$ such that
\begin{equation}\label{e:041}
\rank\left(\begin{array}{cc}A&B\\\tilde A&\tilde B\end{array}\right)=n.
\end{equation}
Let
$$
A^*=\left(\begin{array}{c}A\\\tilde A\end{array}\right)\in
\Mat{n}{m}\qquad\text{and}\qquad B^*=\left(\begin{array}{c}B\\\tilde B\end{array}\right)\in \Mat{n}{n}.
$$
Then, by (\ref{e:041}) and the weak non-planarity of $(F,\mu)$, there is a $x\in V\cap\supp\,\mu$ such that $\det(A^*F(x)+B^*)\not=0$. Therefore, $\rank(A^*F(x)+B^*)=n$. Clearly
$$
A^*F(x)+B^*=\left(\begin{array}{c}AF(x)+B\\[1ex]
\tilde AF(x)+\tilde B
\end{array}\right).
$$
Then, the fact that the rank of this matrix is
$n$ implies that $\rank(AF(x)+B)=r$.
\end{proof}

\bigskip

\noindent\emph{Proof of Lemma~\ref{l:06}.} For $i=1,2$ let $V_i\subset U_i$ be a ball centred in $\supp\,\mu_i$. The ball $V=V_1\times V_2\subset U$ is then centred in $\supp\,\mu$. Let $A\in \Mat{n}{m}$, $B\in \Mat{n}{n}$ and $\rank(A|B)=n$. Our goal is to show that there is a point $(x_1,x_2)\in V\cap\supp\,\mu$ such that $\det(AF(x_1,x_2)+B)\not=0$.

Split $A$ into $A_1\in \Mat{n}{m_1}$ and $A_2\in \Mat{n}{m_2}$ so that $A=(A_1|A_2)$. By (\ref{prod}), we have that $AF(x_1,x_2)+B=A_1F_1(x_1)+A_2F_2(x_2)+B$. Assume for the moment that we have shown that
\begin{equation}\label{e:042}
\exists\ x_2\in V_2\cap\supp\,\mu_2\quad\text{such that}\quad\rank(A_1|A_2F_2(x_2)+B)=n.
\end{equation}
Then, since $(F_1,\mu_1)$ is weakly non-planar, there would be an
$x_1\in V_1\cap\supp\,\mu_1$ such that $\det\big(A_1F_1(x_1)+(A_2F_2(x_2)+B)\big)\not=0$ and the proof would be complete. Thus, it remains to show (\ref{e:042}).

\medskip

Let $r=\rank(A_2|B)$. Using the Gauss method eliminate the last
$n-r$ rows from $(A_2|B)$. This means that without loss of generality we can assume that $(A_1|A_2|B)$ is of the following form
$$
(A_1|A_2|B)=\left(\begin{array}{c|c|c} * & A_2' & C \\ \hline A_1'
& 0 & 0\end{array}\right),
$$
where $A_1'\in \Mat{n-r}{m_1}$, $A_2'\in \Mat{r}{m_2}$ and $C\in
\Mat{r}{n}$. Observe that $\rank(A_2'|C)=r$. Since $\rank(A|B)=n$, we necessarily have that $\rank A_1'=n-r$. Now verify that
\begin{equation}\label{e:043}
(A_1|A_2F_2(x_2)+B)=\left(\begin{array}{c|c} * & A_2'F_2(x_2) + C \\ \hline
A_1' & 0 \end{array}\right)
\end{equation}
Since $\rank(A_2'|C)=r$ and $(F_2,\mu_2)$ is weakly non-planar, by Lemma~\ref{l:07}, there is an $x_2\in V_2\cap\supp\,\mu_2$ such that
$\rank(A_2'F_2(x_2) + C)=r$. This together with the fact that $\rank
A_1'=n-r$ immediately implies that matrix (\ref{e:043}) is of rank $n$. Thus (\ref{e:042}) is established and the proof is complete.
\proofend

\

Using Lemmas~\ref{l:06} alongside \cite[Lemma~2.2]{Kleinbock-Tomanov-07:MR2314053} and \cite[Theorem~2.4]{Kleinbock-Lindenstrauss-Weiss-04:MR2134453} one relatively straightforwardly obtains the following generalisations of Theorem~6.3 from \cite{Kleinbock-Margulis-Wang-09}.

\begin{theorem}\label{t4}
For $i=1,\dots,l$ let an open subset $U_i$ of $\R^{d_i}$, a continuous map $F_i:U_i\to \Mat{m_i}{n}$ and a Federer measure $\mu_i$ on $U_i$\/ be given. Assume that for every $i$ the pair $(F_i,\mu_i)$ is good and weakly non-planar. Let $\mu=\mu_1\times\dots\times \mu_l$ be the product measure on $U=U_1\times\dots\times U_l$, $m=m_1+\dots+m_l$ and let $F:U\to\Mat{m}{n}$ be given by
\begin{equation}\label{prod2}
F(x_1,\dots,x_l)\df \left(\begin{array}{c} F_1(x_1)\\ \vdots \\ F_l(x_l)\end{array}\right).
\end{equation}
Then ~{\rm(a)} $\mu$ is Federer, ~{\rm(b)} $(F,\mu)$ is good and ~{\rm(c)} $(F,\mu)$ is weakly non-planar.
\end{theorem}

A similar analogue can be deduced from Lemma~\ref{l:08} for the transpose of (\ref{prod2}).

\section{Inhomogeneous and weighted extremality}\label{inhomsec}

\subsection{Inhomogeneous approximation}

In the inhomogeneous case, instead of the systems of linear forms $\vv q\mapsto Y\vv q$ given by $Y\in\Mat{m}{n}$, one considers systems of affine forms $\vv q\mapsto Y\vv q+\vv z$ given by the pairs $(Y;\vv z)$, where $Y\in\Mat{m}{n}$ and $\vv z\in\R^m$. The homogeneous case corresponds to $(Y;\vv z)=(Y; \vv0)$. Let us say that $(Y; \vv z)$ is \emph{VWA} (\emph{very well approximable}) if there exists $\ve>0$ such that for arbitrarily large $Q>1$ there are $\vv
q\in\Z^n\bnz$ and $\vv p\in\Z^m$ satisfying
\begin{equation}\label{e:002inh}
 \|Y\vv q+\vv z-\vv p\|^m<Q^{-1-\ve}\qqand\|\vv q\|^n\le Q\,.
\end{equation}
Let us say that $(Y; \vv z)$ is \emph{VWMA} (\emph{very well multiplicatively approximable}) if there exists $\ve>0$ such that for arbitrarily large $Q>1$ there are $\vv
q\in\Z^n\bnz$ and $\vv p\in\Z^m$ satisfying
\begin{equation}\label{e:004inh}
 \Pi(Y\vv q+\vv z-\vv p)<Q^{-1-\ve}\qqand\PI_+(\vv q)\le Q\,.
\end{equation}
The above definitions are consistent with those used in other papers (see, e.g., \cite{Beresnevich-Velani-10:MR2734962, Bugeaud-Mult}). It is easy to see that in the homogeneous case ($\vv z=\vv0$) these definitions are equivalent to those given in \S\ref{intro}. Note that, in general, $(Y; \vv z)$ is VWA if either $Y\vv q+\vv z\in\Z^m$ for some $\vv q\in\Z^n\bnz$, or there is $\ve>0$ such that the inequality
\begin{equation}\label{e:002inhX}
\|Y\vv q+\vv z-\vv p\|^m<\|\vv q\|^{-(1+\ve)n}
\end{equation}
holds for infinitely many $\vv q\in\Z^n$ and $\vv p\in\Z^m$.
Similarly, $(Y; \vv z)$ is VWMA if either $Y\vv q+\vv z$ has an integer coordinate for some $\vv q\in\Z^n\bnz$, or there is $\ve>0$ such that the inequality
\begin{equation}\label{e:004inhX}
\Pi(Y\vv q+\vv z-\vv p)<\Pi_+(\vv q)^{-1-\ve}
\end{equation}
holds for infinitely many $\vv q\in\Z^n$ and $\vv p\in\Z^m$.

\medskip

One says that a measure $\mu$ on $\Mat{m}{n}$ is \emph{inhomogeneously extremal}  (resp., \emph{inhomogeneously strongly extremal}) if for every $\vv z\in\R^m$ the pair $(Y; \vv z)$ is VWA (resp., VWMA) for $\mu$-almost all $Y\in\Mat{m}{n}$. This property holds e.g.\ for Lebesgue measure on $\Mat{m}{n}$ as an easy consequence of the Borel-Canteli Lemma -- see also \cite{Schmidt-1964} for a far more general result. Clearly, any inhomogeneously (strongly) extremal measure $\mu$ is (strongly) extremal. However, the converse is not generally true. For example, Remark~2 in \cite[p.\,826]{Beresnevich-Velani-10:MR2734962} contains examples of lines in $\Mat{2}{1}$ that are strongly extremal but \textsc{not} inhomogeneously strongly extremal. More to the point, {\em no atomic measure can be inhomogeneously extremal}. This readily follows from the fact that for any extremal $Y$ and $v>1$ the set
$$
\cW_Y(v):=\left\{\vv z\in[0,1)^m:\begin{array}{l}
\|Y\vv q+\vv z-\vv p\|^m<\|\vv q\|^{-vn}\text{ holds for}\\[0.5ex]
\text{infinitely many $\vv q\in\Z^n\bnz$ and $\vv p\in\Z^m$}
\end{array}\right\}
$$
is non-empty, and in fact has Hausdorff dimension
$$
\dim\cW_Y(v)=\frac{m}{v}.
$$
The proof of this fact is analogous to that of Theorem~6 from \cite{BugChev} and will not be considered here. The extremality of $Y$ is not necessary to ensure that $\cW_Y(v)\neq\emptyset$. For example, using the effective version of Kronecker's theorem \cite[Theorem~VI,~p.\,82]{Cassels-1957} and the Mass Transference Principle of \cite{Beresnevich-Velani-06:MR2259250} one can easily show the following: \emph{if for some $\ve>0$ inequality \eqref{e:002} has only finitely many solutions $\vv q\in\Z^n$ and $\vv p\in\Z^m$, then $\dim \cW_Y(v)>0$} for any $v>1$.

\bigskip

The main goal of this section is to prove an inhomogeneous generalisation of Theorem~\ref{t2} (see Corollary~\ref{inhom_cor} below). This is based on establishing an inhomogeneous transference akin to Theorem~1 in \cite{Beresnevich-Velani-10:MR2734962}. In short, the transference enables us to deduce the inhomogeneous (strong) extremality of a measure once we know it is (strongly) extremal. As we have discussed above, such a transference is impossible for arbitrary measures and would require some conditions on the measures under consideration. In \cite{Beresnevich-Velani-10:MR2734962}, the notion of contracting measures on $\Mat{m}{n}$ has been introduced and used to establish such a transference. Our following result makes use of the notion of good and non-planar rows which is much easier to verify, thus simplifying and in a sense generalising the result of \cite{Beresnevich-Velani-10:MR2734962}.

\begin{theorem}\label{t2inhom}
Let $U$ be an open subset of\/ $\R^d$, $\mu$ a Federer measure on $U$ and $F:U\to\Mat{m}{n}$ a continuous map. Let $F_j:U\to\R^n$ denote the $j$-th row of $F$.
Assume that the pair $(F_j,\mu)$ is good and non-planar for each $j$. Then we have the following two equivalences
\begin{equation}\label{vbA}
\text{$F_*\mu$ is extremal} \iff \text{$F_*\mu$ is inhomogeneously extremal}\,,~~~~~
\end{equation}
\begin{equation}\label{vbB}
\text{$F_*\mu$ is strongly extremal} \iff \text{$F_*\mu$ is inhomogeneously strongly extremal.}
\end{equation}
\end{theorem}

\noindent Observe that $(F_j,\mu)$ is good and non-planar for each $j$ whenever $(F,\mu)$ is good and weakly non-planar. Hence, Theorems~\ref{t2} and \ref{t2inhom} imply the following

\begin{corollary}\label{inhom_cor}
Let $U$ be an open subset of $\R^d$, $\mu$ a Federer measure on $U$ and $F:U\to\Mat{m}{n}$ a continuous map such that
$(F,\mu)$ is {\rm(i)} good, and {\rm(ii)} weakly  non-planar. Then
$F_*\mu$ is inhomogeneously strongly extremal.
\end{corollary}

\subsection{Weighted approximation}

Weighted extremality is a modification of the standard (non-multiplicative) case obtained by introducing weights of approximation for each linear form. Formally, let $\vv r=(r_1,\dots,r_{m+n})$ be an $(m+n)$-tuple of real numbers such that
\begin{equation}\label{vb+4}
r_i\ge0\quad(1\le i\le m+n)\qqand r_1+\ldots+r_m=r_{m+1}+\ldots+r_{m+n}=1.
\end{equation}
One says that $(Y; \vv z)$ is $\vv r$-\emph{VWA} ($\vv r$-\emph{very well approximable}) if there exists $\ve>0$ such that for arbitrarily large $Q>1$ there are $\vv q\in\Z^n\bnz$ and $\vv p\in\Z^m$ satisfying
\begin{equation}\label{vb+5}
 |Y_j\vv q+z_j-p_j|<Q^{-(1+\ve)r_j}\quad(1\le j\le m)\qqand|q_i|< Q^{r_{m+i}}\quad(1\le i\le n)\,,
\end{equation}
where $Y_j$ is the $j$-th row of $Y$. A measure $\mu$ on $\Mat{m}{n}$ will be called \emph{$\vv r$-extremal} if $(Y; \vv 0)$ is $\vv r$-VWA for $\mu$-almost all $Y\in\Mat{m}{n}$; a measure $\mu$ on $\Mat{m}{n}$ will be called \emph{inhomogeneously $\vv r$-extremal} if for every $\vv z\in\R^m$ the pair $(Y; \vv z)$ is $\vv r$-VWA for $\mu$-almost all $Y\in\Mat{m}{n}$.

It is readily seen that $(Y;\vv z)$ is VWA if and only if it is $(\tfrac1m,\dots,\tfrac1m,\tfrac1n,\dots,\tfrac1n)$-VWA. Thus, (inhomogeneous) extremality is a special case of (inhomogeneous) $\vv r$-extremality. In fact, the strong extremality is also encompassed by $\vv r$-extremality as follows from the following

\begin{lemma}\label{lemV}
$(Y;\vv z)$ is VWMA $\iff$ $(Y;\vv z)$ is $\vv r$-VWA for some $\vv r$ satisfying \eqref{vb+4}.\\ Furthermore, each VWMA pair $(Y;\vv z)$ is $\vv r$-VWA for some $\vv r\in\Q^{m+n}$.
\end{lemma}

Although the argument given below has been used previously in one form or another, the above equivalence is formally new even in the `classical' case $\vv z = \vv0$ and $\min\{m,n\} = 1$.

\bigskip

\begin{proof}
The sufficiency is an immediate consequence of the obvious fact that \eqref{vb+5} implies \eqref{e:004inh}.
For the necessity consider the following two cases.

\medskip

\noindent\textsf{Case (a)}: There exists $\vv q\in\Z^n\bnz$ and ${j_0}$ such that $Y_{j_0}\vv q+z_{j_0}=p_{j_0}\in\Z$. Then it readily follows from the definitions that $(Y;\vv z)$ is both VWMA and $\vv r$-VWA with $r_{j_0}=1$, $r_j=0$ for $1\le j\le m,\ j\neq j_0$, and $r_{m+i}=\tfrac1n$ for $1\le i\le n$.

\medskip

\noindent\textsf{Case (b)}: $Y_j\vv q+z_j\not\in\Z$ for all $\vv q\in\Z^n\bnz$ and $1\le j\le m$.
We are given that for some $\ve\in(0,1)$ there are infinitely many $\vv q\in\Z^n\bnz$ and $\vv p\in\Z^m$ satisfying \eqref{e:004inhX}. Without loss of generality we may also assume that
\begin{equation}\label{vb+10}
    \max_{1\le j\le m}|Y_j\vv q+z_j-p_j|<1.
\end{equation}
Let $0<\ve'<\ve$. Fix any positive parameters $\delta$ and $\delta'$ such that
\begin{equation}\label{vb+9}
\frac{1+\ve}{(1+\delta')(1+\ve')}-m\delta\ge 1, \qquad \frac{1}{1+\delta'}+n\delta\le 1\qqand \frac1\delta\in\Z.
\end{equation}
The existence of $\delta$ and $\delta'$ is easily seen.
For each $(\vv q,\vv p)$ satisfying \eqref{e:004inhX} and \eqref{vb+10} define $Q=\Pi_+(\vv q)^{1+\delta'}$ and the unique $(m+n)$-tuple $\vv u=(u_1,\dots,u_{m+n})$ of integer multiples of $\delta$ such that
\begin{equation}\label{vb+6}
\begin{array}{l}
 Q^{-(1+\ve')(u_j+\delta)}\le |Y_j\vv q+z_j-p_j|<Q^{-(1+\ve')u_j}\quad~(1\le j\le m),\\[1ex]
~~~~~~~~~~~~~Q^{u_{m+i}-\delta}\le |q_i|< Q^{u_{m+i}}\quad~~~~~~(1\le i\le n,\ q_i\neq0),\\[1ex]
~~~~~~~~~~~~~~~~~~~~~~~~u_{m+i}=0\quad~~~~~~~~~~~~~(1\le i\le n,\ q_i=0).
\end{array}
\end{equation}
Let $u=\sum_{i=1}^nu_{m+i}$. Then, by \eqref{vb+6}, we have that $Q^{u-n\delta}\le Q^{1/(1+\delta')}\le Q^{u}$. Therefore,
$1/(1+\delta')\le u\le 1/(1+\delta')+n\delta$. By \eqref{vb+9}, we have that
\begin{equation}\label{vb+11}
1/(1+\delta')\le u\le 1.
\end{equation}
 Next, by \eqref{e:004inhX} and \eqref{vb+6},
\begin{equation}\label{vb+8}
 \prod_{j=1}^mQ^{-(1+\ve')(u_j+\delta)}\times Q^{(1+\ve)/(1+\delta')}\le\Pi(Y\vv q+\vv z-\vv p)\times \PI_+(\vv q)^{1+\ve}<1\,.
\end{equation}
Let $\widehat u=\sum_{j=1}^mu_j$. Then, by \eqref{vb+8}, we get
$Q^{-m(1+\ve')\delta}Q^{-(1+\ve')\widehat u}Q^{(1+\ve)/(1+\delta')}<1$, whence
$$
-m(1+\ve')\delta-(1+\ve')\widehat u+(1+\ve)/(1+\delta')<0.
$$
Hence, by \eqref{vb+9}, we get
\begin{equation}\label{vb+12}
\widehat u>(1+\ve)/(1+\delta')(1+\ve')-m\delta\ge 1.
\end{equation}
By \eqref{vb+11}, \eqref{vb+12} and the fact that $\delta^{-1}\in\Z$, we can  find an $(m+n)$-tuple $\vv r$ of integer multiples of $\delta$ satisfying \eqref{vb+4} such that $r_j\le u_j$ for $1\le j\le m$ and $r_{m+i}\ge u_{m+i}$ for $1\le i\le n$. Then, by \eqref{vb+6}, we get that
\begin{equation}\label{vb+13}
\begin{array}{l}
|Y_j\vv q+z_j-p_j|<Q^{-(1+\ve')r_j}\quad(1\le j\le m),\\[1ex]
~~~~~~~~~~~~~~~|q_i|< Q^{r_{m+i}}\quad~~~~~(1\le i\le n).
\end{array}
\end{equation}
This holds for infinitely many $\vv q$, $\vv p$ and arbitrarily large $Q$. Since the components of $\vv r$ are integer multiples of $\delta$, there is only a finite number of choices for $\vv r$. Therefore, there is a $\vv r$ satisfying \eqref{vb+4} such that \eqref{vb+13} holds for some $\vv q\in\Z^n\bnz$ and $\vv p\in\Z^m$ for arbitrarily large $Q$.
The furthermore part of the lemma is also established as, by construction, $\vv r\in\Q^{m+n}$.
\end{proof}

\bigskip

In view of Lemma~\ref{lemV}, Theorem~\ref{t2inhom} is a consequence of the following transference result regarding $\vv r$-extremality.

\begin{theorem}\label{t2inhom2}
Let $U$ be an open subset of\/ $\R^d$, $\mu$ a Federer measure on $U$ and $F:U\to\Mat{m}{n}$ a continuous map. Let $F_j:U\to\R^n$ denote the $j$-th row of $F$. Let $\vv r$ be an $(m+n)$-tuple of real numbers satisfying \eqref{vb+4}.
Assume that the pair $(F_j,\mu)$ is good and non-planar for each $j$. Then
\begin{equation}
\text{$F_*\mu$ is $\vv r$-extremal} \iff \text{$F_*\mu$ is $\vv r$-inhomogeneously extremal}\,.
\end{equation}
\end{theorem}

\medskip

Another consequence of Lemma~\ref{lemV} and Theorems~\ref{t2inhom2} and \ref{t2} is the following

\begin{theorem}\label{t2inhom3}
Let $U$ be an open subset of $\R^d$, $\mu$ a Federer measure on $U$ and $F:U\to\Mat{m}{n}$ a continuous map such that
$(F,\mu)$ is {\rm(i)} good, and {\rm(ii)} weakly  non-planar. Then $F_*\mu$ is inhomogeneously $\vv r$-extremal for any $(m+n)$-tuple $\vv r$ of real numbers satisfying \eqref{vb+4}.
\end{theorem}

For the rest of \S\ref{inhomsec} we will be concerned with proving Theorem~\ref{t2inhom2}. This will be done by using the Inhomogeneous Transference of \cite[\S5]{Beresnevich-Velani-10:MR2734962} that is now recalled.

\subsection{Inhomogeneous Transference framework}\label{IT}

In this section we recall the general framework of Inhomogeneous Transference of \cite[\S5]{Beresnevich-Velani-10:MR2734962}.
Let $\AAA$ and $\TTT$ be two countable indexing sets. For each $\alpha\in\AAA$, $\ttt\in\TTT$ and $\ve>0$ let $\hom_\ttt(\alpha,\ve)$ and $\inh_\ttt(\alpha,\ve)$
be open subsets of $\R^d$ (more generally the framework allows one to consider any metric space instead of $\R^d$). Let $\Psi$ be a set of functions $\psi:\TTT\to\Rp$. Let $\mu$ be a non-atomic finite Federer measure supported on a bounded subset of $\R^d$. The validity of the following two properties is also required.

\medskip

\noindent\textbf{The Intersection Property}.
For any $\psi\in\Psi$ there exists $\psi^*\in\Psi$ such that for all but finitely many $\ttt\in\TTT$
and all  distinct $\alpha$ and $\alpha'$ in $\AAA$ we have that
\begin{equation}\label{e:066}
    \inh_\ttt(\alpha,\p(\ttt))\cap \inh_\ttt(\alpha',\p(\ttt))\subset
\textstyle\bigcup\limits_{\alpha''\in\AAA}\hom_{\ttt}(\alpha'',\psi^*(\ttt))   \ .
\end{equation}

\medskip

\noindent\textbf{The Contraction Property}.
For any $\psi\in\Psi$ there
exists $\psi^+\in\Psi$ and a sequence of positive numbers
$\{k_\ttt\}_{\ttt\in \TTT}$ satisfying
\begin{equation}\label{e:067}
    \sum_{\ttt\in \TTT}k_\ttt<\infty  ,
\end{equation}
 such that for all but finitely  $\ttt\in \TTT$ and all
$\alpha\in \AAA$ there exists a collection $\Cta$ of balls $B$
centred at $\Supp$ satisfying the
following conditions\,{\rm:}
  \begin{equation}\label{e:068}
\textstyle    \Supp\cap\inh_\ttt(\alpha,\psi(\ttt)) \ \subset \
    \bigcup\limits_{B\in\Cta}B\,,
  \end{equation}
    \begin{equation}\label{e:069}
\textstyle        \Supp\cap\bigcup\limits_{B\in\Cta}B \ \subset \ \inh_\ttt(\alpha,\p^+(\ttt))
    \end{equation}
    and
    \begin{equation}\label{e:070}
        \mu\big(5B\cap\inh_\ttt(\alpha,\psi(\ttt))\big)\ \le  \ k_\ttt \,  \mu(5B) \ .
    \end{equation}

\medskip

For $\psi\in\Psi$, consider the $\limsup$ sets
\begin{equation}\label{e:065}
\textstyle \La_\hom(\psi\,)=\limsup\limits_{\ttt \in \TTT}\bigcup\limits_{\alpha\in\AAA}\hom_\ttt(\alpha,\p(\ttt))
 \qand
 \La_\inh(\psi\,)=\limsup\limits_{\ttt \in \TTT}\bigcup\limits_{\alpha\in\AAA}\inh_\ttt(\alpha,\p(\ttt))\,.
\end{equation}
The following statement from \cite{Beresnevich-Velani-10:MR2734962} will be all that we need to give a proof of Theorem~\ref{t2inhom2}.

\begin{theorem}[Theorem~5 in \cite{Beresnevich-Velani-10:MR2734962}]\label{ITP}
Suppose $\AAA$, $\TTT$, $\hom_\ttt(\alpha,\ve)$, $\inh_\ttt(\alpha,\ve)$, $\Psi$ and $\mu$ as above are given and the intersection and contraction properties are satisfied. Then
\begin{equation}\label{e:071}
\forall\ \psi\in\Psi\ \ \mu(\La_\hom(\psi))=0    \
\qquad\Longrightarrow\qquad
\forall\ \psi\in\Psi\ \  \mu(\La_\inh(\psi))=0 .
\end{equation}
\end{theorem}

\subsection{Proof of Theorem~\ref{t2inhom2}}

While proving Theorem~\ref{t2inhom2} there is no loss of generality in assuming that $r_1,\dots,r_m>0$ as otherwise we would consider the smaller system of forms that correspond to $r_j>0$.

From now on fix any $\vv z\in\R^m$. With the aim of using Theorem~\ref{ITP} define $\TTT=\Z_{\ge0}$,
$\AAA=(\Z^n\bnz)\times\Z^m$
and $\Psi=(0,+\infty)$, that is the functions $\psi\in\Psi$ are constants.
Further for $t\in\TTT$, $\alpha=(\vv q,\vv p)\in\cA$ and $\ve>0$, let
\begin{equation}\label{I}
\rI_t(\alpha,\ve)=\left\{x\in U :
\begin{array}{l}
|F_j(x)\vv q+z_j-p_j|< \frac12\cdot2^{-(1+\ve)r_{j}t}\quad (1\le j\le m)\\[0.5ex]
~~~~~~~~~~~~~~~~~~~|q_i| < \frac12\cdot2^{r_{m+i}t}\quad ~~~~~(1\le i\le n)
\end{array}
\right\}
\end{equation}
and
\begin{equation}\label{H}
\rH_t(\alpha,\ve)=\left\{x\in U\,:\,
\begin{array}{l}
|F_j(x)\vv q-p_j|< 2^{-(1+\ve)r_{j}t}\quad (1\le j\le m)\\[0.5ex]
~~~~~~~~~~~~~|q_i| < 2^{r_{m+i}t}\quad ~~~~~(1\le i\le n)
\end{array}
\right\}.
\end{equation}

\begin{proposition}\label{prop5}
Let $x\in U$ . Then
\begin{itemize}
  \item[{\rm(i)}] $(F(x);\vv z)$ is $\vv r$-VWA $\iff$ $x\in\La_\inh(\psi)$ for some $\psi>0;$
  \item[{\rm(ii)}] $(F(x);\vv0)$ is $\vv r$-VWA $\iff$ $x\in\La_\hom(\psi)$ for some $\psi>0$.
\end{itemize}
\end{proposition}

\medskip

Proposition~\ref{prop5} and Theorem~\ref{ITP} would imply Theorem~\ref{t2inhom2} upon establishing the intersection and contraction properties. While postponing the verification of these properties till the end of the section, we now give a proof of Proposition~\ref{prop5}.

\bigskip

\begin{proof}
We consider the proof of part (i) as that of part (ii) is similar (and in a sense simpler).
Assume that $(F(x),\vv z)$ is $\vv r$-VWA. Then
there exists $\ve>0$ such that for arbitrarily large $Q>1$ there are $\vv q\in\Z^n\bnz$ and $\vv p\in\Z^m$ satisfying \eqref{vb+5} with $Y=F(x)$. For each such $Q$ define $t\in\N$ such that $2^{t-1}<2^{1/r'}Q\le 2^{t}$, where $r'=\min\{r_{m+i}>0:1\le i\le n\}$. Let $0<\psi<\ve$. Then, by \eqref{vb+5} with $Y=F(x)$, we have that
$$
 |F_j(x)\vv q+z_j-p_j|<2^{(1+\ve)r_j}2^{-(1+\ve)r_jt}<\tfrac12\cdot2^{-(1+\p)r_jt}\qquad\text{for $1\le j\le m$}
$$
when $t$ is sufficiently large. Here we use the fact that $r_j>0$. Also when $r_{m+i}>0$ we have that $Q^{r_{m+i}}\le \tfrac12\cdot 2^{r_{m+i}t}$. This is a consequence of the definition of $t$. Hence by \eqref{vb+5} with $Y=F(x)$, we have that
\begin{equation}\label{vb+19}
|q_i|< \tfrac12\cdot 2^{r_{m+i}t}\qquad\text{for $1\le i\le n$}
\end{equation}
when $r_{m+i}>0$. If $r_{m+i}=0$, then we have that $|q_i|<Q^{r_{m+i}}=1$. Since $q_i\in\Z$ we necessarily have that $q_i=0$. Consequently \eqref{vb+19} also holds when $r_{m+i}=0$. Thus, $x\in\inh_t(\alpha,\p)$ and furthermore this holds for infinitely many $t$. Therefore, $x\in\La_\inh(\psi)$. The sufficiency is straightforward because the fact that $x\in\La_\inh(\psi)$ means that with $\ve=\p$ for arbitrarily large $Q=2^t$ $(t\in\N)$ there are $\vv q\in\Z^n\bnz$ and $\vv p\in\Z^m$ satisfying \eqref{vb+5} with $Y=F(x)$. Hence $(F(x);\vv z)$ is $\vv r$-VWA.
\end{proof}

\bigskip

\bigskip

\noindent\textit{Verifying the intersection property.}
Take any $\psi\in\Psi$ and distinct $\alpha=(\vv q,\vv p)$ and $\alpha'=(\vv q',\vv p')$ in $\cA$.
Take any point $x \in \ \inh_t(\alpha,\p)\cap\inh_t(\alpha',\p)$.
It means that
\begin{equation}\label{vb+7}
\begin{array}{ll}
|F_j(x)\vv q+z_j-p_j|< \tfrac12\cdot2^{-(1+\p)r_{j}t},\qquad& |q_i| <\tfrac12\cdot 2^{r_{m+i}t},\\[1ex]
|F_j(x)\vv q'+z_j-p'_j|<\tfrac12\cdot2^{-(1+\p)r_{j}t}, &
|q'_i| <\tfrac12\cdot 2^{r_{m+i}t}
\end{array}
\end{equation}
for $1\le j\le m$ and $1\le i\le n$.
Let $\alpha''=(\vv q'',\vv p'')$, where $\vv p''=\vv p-\vv p'\in\Z^m$ and $\vv q''=\vv q-\vv q'\in\Z^n$.
Using \eqref{vb+7} and the triangle inequality we obtain that
\begin{equation}\label{e:073}
|F_j(x)\vv q''-p''_j|< 2^{-(1+\p)r_{j}t}\qqand
|q''_i| < 2^{r_{m+i}t}
\end{equation}
for $1\le j\le m$ and $1\le i\le n$.
If $\vv q=\vv q'$, then $\vv p''\neq\vv0$ (because $\alpha\neq\alpha'$) and $|p''_j|< 2^{-(1+\p)r_{j}t}\le1$. Since $p''_j\in\Z$ and $|p''_j|<1$ we must have that $p''_j=0$ for all $j$, contrary to $\vv p''\neq\vv0$. Therefore, we must have that $\vv q''\not=\vv0$ and so $\alpha''\in\cA$.
By \eqref{e:073}, we get that
$x\in \hom_{t}(\alpha'',\psi)$. This verifies the intersection property with $\p^*=\p$.

\bigskip

\noindent\textit{Verifying the contraction property.}
Since $(F_j,\mu)$ is good for each $j$, for almost every $x_0\in \Supp\cap U$ there exist positive $C_j$ and $\alpha_j$ and a ball $V_j$ centred at $x_0$ such that for each $\vv q\in\R^n$, $p\in\R$ and $1\le j\le m$ the function $F_j(x)\vv q+p$ is $(C_j,\alpha_j)$-good on $V_j$ with respect to $\mu$. Let $C=\max C_j$, $\alpha=\min\alpha_j$ and $V=\cap_jV_j$. Then
for each $j$ $(1\le j\le m)$, $\vv q\in\R^n$ and $p\in\R$ the function
\begin{equation}\label{vb+21}
F_j(x)\vv q+p\quad\text{is $(C,\alpha)$-good on $V$ with respect to $\mu$}.
\end{equation}
Since the balls $V$ obtained this way cover $\mu$-almost every point of $U$ without loss of generality we will assume that $U=V$ and that $\Supp\subset U$ within our proof of Theorem~\ref{t2inhom2}. Also since $\mu$ is a Radon measure, without loss of generality we can assume that $\mu$ is finite.

Since $(F_j,\mu)$ is non-planar for each $j$, we have that
$$
d_j(\vv q,p)\ \stackrel{\rm def}{=}\ \frac{\|F_j(x)\vv q+p\|_{\mu,U}}{\|\vv q\|}>0
$$
for each $\vv q\in\R^n\bnz$ and $p\in\R$. The quantity $d_j(\vv q,p)$ is the distance of the furthest point of $F_j(\Supp)$ from the hyperplane $\vv y\cdot\vv q+p=0$. Obviously, this is a continuous function of $\vv q$ and $p$. Hence it is bounded away from zero on any compact set, in particular, on $\{\vv q:\|\vv q\|=1\}\times [-N,N]$, where it takes its minimum for a sufficiently large $N$. Hence there is an $r_0>0$ such that
\begin{equation}\label{vb+22}
\|F_j(x)\vv q+p\|_{\mu,U}\ge r_0\|\vv q\|
\end{equation}
for all $\vv q\in\R^n\bnz$, $p\in\R$ and $1\le j\le m$.

Let $\psi>0$ and $0<\psi^+<\p$.
By \eqref{vb+22} and the assumption that $\min_{1\le j\le m}r_j>0$, for sufficiently large $t$ we have that
\begin{equation}\label{e:020}
\Supp\not\subset\rI_t(\alpha,\psi^+).
\end{equation}
We now construct a collection $\cC_{t,\alpha}$ required by the contraction property, where $t\in\Z_{\ge0}$ is sufficiently large and $\alpha=(\vv q,\vv p)\in\Z^n\bnz\times\Z^m$. If $\Supp\cap \rI_t(\alpha,\psi)=\emptyset$,  then taking $C_{t,\alpha}=\emptyset$ does the job. Otherwise, for each $x\in\Supp\cap\rI_t(\alpha,\psi)$ take any ball $B'\subset\rI_t(\alpha,\psi)$ centred at $x$. Clearly, this is possible because $\rI_t(\alpha,\psi)$ is open. Since $\p^+<\p$, we have that $\rI_t(\alpha,\psi)\subset \rI_t(\alpha,\psi^+)$. Therefore, by \eqref{e:020}, there exists $\tau\ge1$ such that
\begin{equation}\label{e:022}
5\tau B'\cap\Supp\not\subset\rI_t(\alpha,\psi^+) \qqand \tau
B'\cap\Supp\subset\rI_t(\alpha,\psi^+)\,.
\end{equation}
Let $B=B(x)=\tau B'$. By the left hand side of \eqref{e:022}, there exists $j\in\{1,\dots,m\}$ and $x_0\in\Supp\cap 5B$ such that
$$
|f(x_0)|\ge \tfrac12\cdot2^{-(1+\p^+)r_{j}t},\qquad\text{where $f(x)=F_j(x)\vv q+z_j-p_j$.}
$$
Hence
$\|f\|_{\mu,5B}\ge \tfrac12\cdot2^{-(1+\p^+)r_{j}t}$.
Observe that
$$
5B\cap\inh_t(\alpha,\psi)\subset \big\{x\in 5B:|f(x)|<\tfrac12\cdot2^{-(1+\p)r_{j}t}\big\}.
$$
Then, since $f$ is $(C,\alpha)$-good, we have that
$$
\begin{array}{rcl}
\mu\big(5B\cap\inh_t(\alpha,\psi)\big) & \le &
\mu\big\{x\in 5B:|f(x)|<\tfrac12\cdot2^{-(1+\p)r_{j}t}\big\} \\[2ex]
& \le & C\left(\frac{\textstyle\tfrac12\cdot2^{-(1+\p)r_{j}t}}{\textstyle\|f\|_{\mu,5B}}\right)^\alpha\mu(5B)
 \le  C\left(\frac{\textstyle\tfrac12\cdot2^{-(1+\p)r_{j}t}}{\textstyle\tfrac12\cdot2^{-(1+\p^+)r_{j}t}}\right)^{\alpha} \mu(5B) \\[3ex]
& \le & C\cdot 2^{-(\p-\p^+)r_j\alpha t}\,\mu(5B)= k_t\,\mu(5B)
\end{array}
$$
where
$$
~~k_t=C\cdot 2^{-(\p-\p^+)r_j\alpha t}.
$$
Clearly, \eqref{e:067} holds. Also, by construction, conditions \eqref{e:068}--\eqref{e:070} are satisfied for the collection $C_{t,\alpha}:=\{B(x):x\in\Supp\cap\rI_t(\alpha,\psi)\}$. This completes the proof of Theorem~\ref{t2inhom2}.

\section{Final remarks}

\subsection{Checking weak non-planarity} The condition of weak non-planarity of pairs $(F,\mu)$ has been demonstrated in this paper to have may nice and natural features. But how one can in general show that a given pair is weakly non-planar? This question is tricky even in the analytic category.
If $\min\{m,n\}=1$ and
$\cM$ is immersed into
$\R^n$ by an analytic map $\vv f=(f_1,\dots,f_n)$, its non-planarity can be verified by taking partial derivatives of $\vf$, i.e.\ via  (\ref{nd}). However, when $\min\{m,n\}>1$ finding an algorithmic way to verify weak non-planarity seems to be an open problem.

Here is a specific example: {\it a matrix version of Baker's problem\/}. Let $m,k\in\N$ and $n=mk$. Let
$$
\cM=\{(X,\dots,X^n)\in M_{m,mn}:X\in M_{m,m}\}.
$$
It seems reasonable to conjecture that $\cM$ is strongly extremal. In the case $k=1$ this problem reduces to Baker's original problem on strong extremality of the Veronese curves. When $m=n=2$ the manifold $\cM$ happens to be non-planar and so weakly non-planar. This is easily verified by writing down all the minors of $(X,X^2)$. It is however unclear how to verify (or disprove) that $\cM$ is weakly non-planar (or possibly strongly non-planar) for arbitrary $m$ and $n$. Note also that the extremality of this manifold has been established in \cite{dima geom}, however the argument is not powerful enough to yield strong extremality.

\subsection{Beyond weak non-planarity}  Let $\cM$ be an analytic manifold in $\Mat{m}{n}$, and let
$$
\cH({\cM)}=\bigcap_{\substack{\cH\in\mathbf{H}_{m,n}\\ \cM\subset\cH}}\cH.
$$
If $\cM\not\subset\cH$ for every $\cH\in\mathbf{H}_{m,n}$, then, by definition, we let $\cH({\cM)}=\Mat{m}{n}$.
In the case  $\min\{m,n\}=1$ the set $\cH(\cM)$ is simply an affine subspace of $\R^m$ or $\R^n$, depending on which of the dimensions is $1$. It is shown in \cite{Kleinbock-03:MR1982150} that {\em if\, $\min\{m,n\}=1$ then $\cM$ is (strongly) extremal if and only if so is $\cH(\cM)$.} A natural question is whether a similar characterisation of analytic (strongly) extremal manifolds in $\Mat{m}{n}$ is possible in the case of arbitrary $(m,n)$.


\subsection{Hausdorff dimension} Another natural challenge is to investigate the Hausdorff dimension of
the exceptional sets of points lying on a non-planar manifold in
$\Mat{m}{n}$ such that (\ref{e:002}) (or (\ref{e:004})~) has infinitely
many solutions (for some fixed $\ve>1$). The upper bounds for
Hausdorff dimension are not fully understood even in the case of
manifolds in $\R^n$ -- see \cite{Beresnevich-Bernik-Dodson-02:MR2069553, Beresnevich-Velani-07:MR2285737, Bernik-1983a, BernikDodson-1999}. However, there has been great success with establishing lower bounds -- see
\cite{Beresnevich-SDA1, Beresnevich-Dickinson-Velani-06:MR2184760, Beresnevich-Dickinson-Velani-07:MR2373145, Beresnevich-Velani-07:MR2285737, DickinsonDodson-2000a}.

\subsection{Khintchine-Groshev type theory} The fact that Lebesgue measure on $\Mat{m}{n}$ is extremal can be thought of as a special case of the convergence part of the Khintchine-Groshev theorem.
Specifically, generalizing \eqref{e:002}, for a function  $\psi$ one says that $Y\in \Mat{m}{n}$ is $\psi$-approximable if  the inequality
\begin{equation}\label{psi}
\|Y\vv q-\vv p\|<\psi(\|\vv q\|)
\end{equation}
holds for infinitely many $\vv q\in\Z^n$ and $\vv p\in\Z^m$. A result of Groshev (1938), generalizing Khintchine's earlier work, states that for non-increasing $\psi$, Lebesgue almost no (resp., almost all) $Y\in \Mat{m}{n}$ are $\psi$-approximable if the sum
\eq{kgt}{\sum_{k = 1}^\infty k^{n-1}\psi(k)^m}
converges (resp., diverges). The convergence part straightforwardly follows from the Borel-Cantelli Lemma and does not require the monotonicity of $\psi$; in the divergence part the monotonocity assumption was recently removed in \cite{Beresnevich-Velani-10:MR2576284} in all cases except $m = n = 1$, where it is known to be necessary.

Proving similar results for manifolds of $\Mat{m}{1}$ and  $\Mat{1}{n}$ has been a fruitful activity, see the monograph \cite{BernikDodson-1999} for some earlier results, and \cite{Beresnevich-02:MR1905790, Beresnevich-SDA1, Beresnevich-Bernik-Kleinbock-Margulis-02:MR1944505, Beresnevich-Dickinson-Velani-07:MR2373145, Bernik-Kleinbock-Margulis-01:MR1829381} for more recent developments.
It seems natural to conjecture that, for a monotonic $\psi$, almost no (resp., almost all) $Y$ on a weakly non-planar analytic submanifold of  $ \Mat{m}{n}$ are $\psi$-approximable if the sum \equ{kgt}
converges (resp., diverges). Presently no results are known when $\min\{m,n\} >1$ except for $\psi$ given by the right hand side of \eqref{e:002}, or for the manifold being the whole space $\Mat{m}{n}$. One can also study a multiplicative version of the problem, which is much more challenging and where much less is known, see \cite{Beresnevich-Velani-07:MR2285737}.

\subsection{Other spaces} The analogue of the Baker-Sprind\v zuk conjecture has been established in $\C^n$, $\Q_p^n$ and in products of archimedean and non-archimedean spaces --
see, e.g., \cite{Kleinbock-04:MR2094125, Kleinbock-Tomanov-07:MR2314053}. It would be reasonable to explore
similar generalisations of Theorem~\ref{t2}.

{\footnotesize

\

\noindent
{\sc VB: University of York, Heslington, York, YO10 5DD, England}\\
\hspace*{5.5ex}{\it E-mail}\,:~~ \verb|victor.beresnevich@york.ac.uk|\\[2ex]
{\sc DK: Brandeis University, Waltham MA 02454-9110}\\
\hspace*{5.5ex}{\it E-mail}\,:~~ \verb|kleinboc@brandeis.edu|\\[2ex]
{\sc GM: Yale University, New Haven, CT 06520}\\
\hspace*{5.5ex}{\it E-mail}\,:~~ \verb|margulis@math.yale.edu|

}

\end{document}